\documentclass[10pt,a4paper]{amsart}

\usepackage[top=3cm, bottom=2.5cm, left=2.5cm, right=2.5cm]{geometry}
\usepackage{amssymb}
\usepackage{amsmath}
\usepackage{amscd}
\usepackage{mathbbol}
\usepackage[latin1]{inputenc}
\usepackage{epic}
\usepackage{eepic}
\usepackage{pifont}
\usepackage[usenames,dvipsnames]{color}

\usepackage[font=small,skip=0pt]{caption}
\usepackage{subcaption}

\usepackage{rotating}
\usepackage{mathtools}
\usepackage{setspace}

\usepackage[section]{algorithm}
\usepackage{algorithmicx,algpseudocode}

\def \R  {{\mathbb R}}

\def \Z  {{\mathbb Z}}

\def \C  {{\mathbb C}}
\def \N  {{\mathbb N}}
\def \CP {{{\mathbb C}{\mathbb P}}}

\def \CP {{\mathbb C}{\mathbb P}}
\def \t  {{\mathfrak t}}

\def \bfZ {{\mathbf Z}}

\def \calE {{\mathcal E}}
\def \calJ {{\mathcal J}}

\newcommand{\algt}{\ensuremath{\mathfrak{t}}}

\def    \eps    {\epsilon}

\def    \half   {{\frac{1}{2}}}

\def    \ol {\overline}

\def \tN {{\widetilde{N}}}

\DeclareMathOperator \defect {defect}

\DeclareMathOperator \GW {GW}

\newcommand{\acts}{\mathbin{\raisebox{-.5pt}{\reflectbox{\begin{sideways}$\circlearrowleft$\end{sideways}}}}}

\numberwithin{figure}{section}

\numberwithin{equation}{section}
\swapnumbers

\makeatletter
\let\c@equation\c@figure
\makeatother

\makeatletter
\let\c@table\c@figure
\makeatother

\makeatletter
\let\c@algorithm\c@figure
\makeatother

\newtheorem{Lemma}[equation]{Lemma}
\newtheorem{Theorem}[equation]{Theorem}
\newtheorem*{thm*}{Theorem}

\newtheorem{Question*}{Question}

\newtheorem{Proposition}[equation]{Proposition}
\newtheorem{Corollary}[equation]{Corollary}

\newtheorem*{Lemma*}{Lemma}
\newtheorem*{Corollary*}{Corollary}

\theoremstyle{definition}
\newtheorem{Notation}[equation]{Notation}

\newtheorem{Facts}[equation]{Facts}
\newtheorem{Definition}[equation]{Definition}
\newtheorem{Remark}[equation]{Remark}
\newtheorem{noTitle}[equation]{}

\newcommand{\labell}[1] {\label{#1} }% \printname{#1}}

\newcommand{\mute}[1] {}
\begin{document}

\title[Hamiltonian circle actions]
{Circle actions on symplectic four-manifolds}
%\\
%with an Appendix by Tair Pnini}

\author{Tara S. Holm}
\address{Dept.\ of Mathematics, Cornell University, Ithaca, NY  14853-4201, USA}
\email{tsh@math.cornell.edu}

\author{Liat Kessler}
\address{Dept.\ of Mathematics, Physics, and Computer Science, University of Haifa,
at Oranim, Tivon 36006, Israel}
\email{lkessler@math.haifa.ac.il}


\maketitle

\begin{center}
with an Appendix by {\sc Tair Pnini}
\end{center}

\begin{abstract}
We complete the classification of Hamiltonian torus and circle actions on 
symplectic four-dimensional manifolds.  
Following work of Delzant and Karshon,  Hamiltonian circle and $2$-torus actions 
on any fixed simply connected symplectic four-manifold
were characterized by Karshon, Kessler and Pinsonnault.
What remains is to study the case of Hamiltonian actions on blowups of $S^2$-bundles 
over a Riemann surface of positive genus.
These do not admit $2$-torus actions.  In this paper, 
we characterize Hamiltonian circle actions on them. 
We then derive combinatorial results on the existence and counting of these actions.
As a by-product, 
we provide an algorithm that determines the $g$-reduced form of a blowup form.
Our work is a combination of ``soft" equivariant 
and combinatorial techniques, using the momentum map and related data, 
with ``hard" holomorphic techniques, including Gromov-Witten invariants.
\end{abstract}

%%%%%%%%%%%%
\section{Introduction}
%%%%%%%%%%%%

In this paper we complete the classification of Hamiltonian torus and circle actions on symplectic
four-dimensional manifolds.   It is in this dimension that 
there are abundant examples (indeed, Gompf has shown that any 
finitely presented group may be the fundamental group of a four-dimensional
symplectic manifold \cite{Gompf}), and yet powerful holomorphic techniques ($J$-holomorphic curves and Gromov-Witten
invariants) are developed enough to make classification problems tractable.

A symplectic action of a $T^n=(S^1)^n$ on a symplectic manifold $(M,\omega)$ is  {\bf Hamiltonian} if it admits a {\bf moment map}: 
an equivariant smooth map $\Phi \colon M \to {\t}^{*}$ such that the components $\Phi^{\xi}=\langle \Phi,\xi \rangle$ satisfy 
Hamilton's equation
$d \Phi^{\xi}=-\iota(\xi_{M})\omega,$
for all $\xi \in \t$. Here $\xi_M$ is the vector field that generates the action on $M$ of the one-parameter 
subgroup $\{\exp(s \xi) \, |\, s \in \R\}$ of $T$.
The actions we consider are effective and the manifolds are compact and connected.
By the Convexity Theorem 
\cite{At:convexity,GS:convexity},
the image of the momentum map $\Delta =\Phi(M)\subset \algt^*\cong \R^n$ 
is a convex polytope. We say that two Hamiltonian $T$ actions are {\bf equivalent} if they differ by an equivariant 
symplectomorphism and a reparametrization of $T$. 
For an effective 
Hamiltonian action $T^n\acts M$,  
we are forced to have
$
\dim(T)\leq \frac{1}{2}\dim(M).
$
Thus, if $M$ is four-dimensional, the only tori that can act effectively are $S^1$ 
and $T^2$.

Delzant \cite{delzant} has classified symplectic manifolds equipped with a Hamiltonian action of a torus of half the dimension
in terms of their momentum polytopes,
up to equivariant symplectomorphism. Building on work of Audin \cite{audin} and Ahara and Hattori \cite{ahara-hattori},
Karshon \cite{karshon} has classified  symplectic four-manifolds with Hamiltonian $S^1$
actions in terms of decorated graphs.
In particular, the only four-manifolds that admit Hamiltonian $S^1$ actions are 
blowups of $S^2$-bundles over Riemann surfaces and $\CP^2$. \cite[\S\! 6]{karshon}.

Delzant's Theorem and Karshon's 
Theorem leave open questions about
inequivalent actions on the same manifold.  Explicitly, given a compact connected four-dimensional symplectic manifold $(M,\omega)$, what Hamiltonian actions does it admit? 
The difficulty lies in determining exactly which Hamiltonian spaces $S^1\acts (M,\omega)$ have symplectomorphic underlying 
manifolds $(M,\omega)$.
The question of characterizing Hamiltonian $2$-torus and circle actions on simply connected 
symplectic four-manifolds  was answered by Karshon, Kessler and Pinsonnault in \cite{algann,algorithm}.

To complete the characterization of Hamiltonian actions on symplectic four manifolds, it
remains to study the case of Hamiltonian  actions on blowups of $S^2$-bundles over a Riemann surface of positive genus.
These do not admit Hamiltonian $2$-torus actions.  In what follows, we characterize Hamiltonian circle actions on them,
up to (possibly non-equivariant) symplectomorphism (Theorems~\ref{conjcount} and \ref{thmexu}).  
We prove that a reduced symplectic blowup of an irrational ruled symplectic  four-manifold is compatible with all the Hamiltonian circle actions: 
there are no ``exotic actions". 

We give an algorithm in pseudo-code in Appendix~\ref{tair}, written by Tair Pnini,
to count the (inequivalent) Hamiltonian circle actions on blowups of irrational ruled symplectic manifolds.
We use the algorithm to derive results on the existence and number of such actions in Section~\ref{counting}.
Along the way, we also provide an algorithm that determines the reduced form 
corresponding to a symplectic form
on a blowup of a ruled symplectic manifold
(Algorithm~\ref{alg:cremona}).
The proofs combine ``soft" combinatorial techniques and ``hard" holomorphic techniques.

\vskip 0.2in

\noindent {\bf Acknowledgements.} 
Tair Pnini is a student at the University of Haifa at Oranim and she implemented the counting
algorithms for us.  We are indebted to her for her work.
We also thank Daniel Cristofaro-Gardiner, 
Yael Karshon, Dusa McDuff and Martin Pinsonnault for helpful discussions.  We are grateful to the
anonymous referee for remarks which improved the paper.

Holm was
supported in part by  the Simons Foundation under
Grant \#208975  and the National Science Foundation under Grants DMS--1206466 and DMS--1128155.
Any opinions, findings, and conclusions or recommendations expressed in this material are those of the authors and do
not necessarily reflect the views of the National Science Foundation or Simons Foundation.

%%%%%%%%%%%%
\section{Statement of Main Results}\label{se:main results}
%%%%%%%%%%%%

\noindent We now turn to the full statement of our main results.  We begin by setting our notation.

\begin{Notation}\label{basic notions}
Let $(\Sigma,j)$ be a compact connected  Riemann surface endowed with a complex structure $j$.  Up to diffeomorphism, there are two $S^2$-bundles over $\Sigma$. We fix a smooth structure on the trivial bundle 
$\Sigma \times S^2$ and on the non-trivial bundle $M_{\Sigma}$.
We equip $\Sigma \times S^2$ and $M_{\Sigma}$ with a complex structure such that each fiber is a holomorphic sphere. 
We fix basepoints $\ast\in S^2$ and $\ast\in \Sigma$.
For the trivial $S^2$-bundle $\Sigma \times S^2$ over $\Sigma$, we denote 
$F:=[\ast \times S^2], \,\,\, B:=[\Sigma \times \ast],$
 classes in the homology group  $H_2(\Sigma \times S^2;\Z)$. 
When we consider the non-trivial $S^2$-bundle $M_{\Sigma}\stackrel{\pi}{\to}\Sigma$, denote the homology class of the fiber by 
$F=[\pi^{-1}(\ast)]\in H_2(M_\Sigma;\Z)$.
For each $\ell$, the trivial bundle admits a section ${\sigma}_{2\ell}: \Sigma\to  \Sigma \times S^2$ 
whose image $\sigma_{2\ell}(\Sigma)$ has 
even self intersection number $2\ell$. Similarly, for each $\ell$, 
the non-trivial bundle admits a section $\sigma_{2\ell+1}:\Sigma\to M_{\Sigma}$ 
whose image $\sigma_{2\ell+1}(\Sigma)$ has odd self intersection number $2\ell+1$.  
We denote $B_n:=[\sigma_n(\Sigma)] \in H_2$ over $\Z$.
For every $n \in \Z$, we have the classes $B_{n}=B_{-n}+nF.$

For a non-negative integer $k$, denote by $(\Sigma \times S^2)_{k}$ the manifold obtained from the trivial $S^2$-bundle over $\Sigma$ by $k$ complex blowups at $k$ distinct points, and by $(M_{\Sigma})_{k}$ the manifold obtained from the non-trivial $S^2$-bundle over $\Sigma$ by $k$ complex blowups at $k$ distinct points.
Let $E_1, \ldots, E_k$
denote the homology classes of the exceptional divisors. We fix the labeling of the exceptional divisors.

Let $M=(\Sigma \times S^2)_{k}$ or $M=(M_{\Sigma})_{k}$.
We say that a vector $(\lambda_F,\lambda_B;\delta_1,\ldots,\delta_k)$ in $\R^{2+k}$
{\bf encodes} a degree $2$ cohomology class $\Omega \in H^2(M;\R)$
if 
\begin{itemize}
\item $\frac{1}{2\pi} \left< \Omega , F \right> = \lambda_F$;  
\item $\frac{1}{2\pi} \left< \Omega , E_j \right> = \delta_j$
for $j= 1, \ldots, k$; and 
\item Either \begin{itemize}
\item[$\circ$]  $\frac{1}{2\pi} \left< \Omega , B \right> = \lambda_B$ when  $M=(\Sigma \times S^2)_{k}$; or
\item[$\circ$] $\frac{1}{2\pi} \left(\left< \Omega , B_{-1} \right>+\frac{1}{2} \left<\Omega,F \right>\right)=\lambda_B$ when  $M=(M_{\Sigma})_{k}$. 
\end{itemize}
\end{itemize}

For $k \geq 2$, let $M=(\Sigma \times S^2)_{k}$ or $M=(M_{\Sigma})_{k}$.
we say that a cohomology class $\Omega \in H^2(M;\R)$ encoded by a vector  $(\lambda_F,\lambda_B;\delta_1,\ldots,\delta_k)$
is {\bf in $g$-reduced form} or is {\bf $g$-reduced}
if 
\begin{equation}\label{eq:red5}
 \delta_1 \geq \ldots \geq \delta_k\, , \quad  
   \delta_1 + \delta_2 \leq \lambda_F\, ,
   \end{equation}
   and, if $g(\Sigma)=0$ and $M=(\Sigma \times S^2)_{k}$,
   \begin{equation} \label{eq:red6}
\lambda_F \leq \lambda_B \,,
   \end{equation}
   and, if $g(\Sigma)=0$ and $M=(M_{\Sigma})_{k}$,
\begin{equation}\label{eq:red7} 
\bigg(\half \lambda_F+\delta_1\bigg)  \leq \lambda_B.
\end{equation}

\end{Notation}

\begin{Remark}
We can realize  $(M_{S^2})_{k}$ as a complex blowup of the complex projective plane $\CP^2$,
the class $B_{-1}$ with the class $E$ of the
exceptional divisor, and the class $B_{1}$ with the class $L$ of a line $\CP^1$ in $\CP^2$.
In this notation, $F=L-E$.
A cohomology class on $(M_{S^2})_{k}$
encoded by $(\lambda_F,\lambda_B;\delta_1,\ldots,\delta_k)$  may  be identified 
with a cohomology class on a $(k+1)$-fold blowup of $\CP^2$ encoded by 
$(\lambda; \delta,\delta_1,\ldots,\delta_k)$, where
\begin{equation} \labell{eqlamdel}
\lambda_F=\lambda-\delta \mbox{ and } \lambda_B=\half\bigg(\lambda+\delta\bigg).
\end{equation}
In this case, our definition of {\bf $0$-reduced} coincides with the 
notion of {\bf reduced} given in \cite{blowups, algann, algorithm}.  
That is, a class $\Omega \in H^2((M_{S^2})_{k};\R)$ is in $0$-reduced form 
if and only if the vector 
$(\lambda_F,\lambda_B;\delta_1,\ldots,\delta_k)$ encoding it determines, by \eqref{eqlamdel}, a vector 
$(\lambda;\delta,\delta_1,\dots,\delta_k)$  which satisfies
$\delta \geq \delta_1 \geq \ldots \geq \delta_k$ and $\delta+\delta_1+\delta_2 \leq \lambda$. 
In the positive genus case, our notion of $g$-reduced does not coincide with
the notion of reduced given in \cite{Li-Li02, Li-Liu}.  An explanation of the difference
is in Remark~\ref{variation of reduced}.
\end{Remark}

\begin{Definition}\label{def:blowup}
We assume $g(\Sigma)>0$.
A {\bf blowup form} on $(\Sigma \times S^2)_{k}$ or $(M_{\Sigma})_{k}$ is 
a symplectic form for which there exist disjoint 
embedded symplectic spheres (oriented by the symplectic form)  in the 
homology classes $F,E_1, \ldots, E_k$.
\end{Definition}

\begin{Remark} \label{rem:blowup}
Blowing down symplectically along $k$ disjoint $\omega$-symplectic embedded spheres in $E_1,\ldots,E_k$ in  $M=(\Sigma \times S^2)_{k}$, or $M=(M_{\Sigma})_{k}$, with a blowup form $\omega$ yields a symplectic manifold $(M',\omega')$ with an $\omega'$-symplectic embedded sphere $C$ in a simple homology class $F$ with self intersection zero.
Then,  by \cite[Proposition 4.1]{structure}, there is a smooth $S^2$-fibration of the obtained manifold $M'$ over a compact $2$-manifold $\Sigma'$ which is compatible with $\omega'$ (i.e., is nondegenerate on each fiber), and has a fiber equal to $C$.
By homology considerations, $\Sigma'=\Sigma$, and the obtained manifold is a trivial $S^2$-bundle over $\Sigma$ if $M=(\Sigma \times S^2)_{k}$ and a nontrivial $S^2$-bundle over $\Sigma$ if $M=({M_{\Sigma}})_{k}$.  
Moreover, by \cite[\S 6]{ML}, the two quantities
$[\omega']^2$ and $\langle [\omega'],F \rangle$ uniquely determine the symplectic
form on $M'$, up to isotopy; see \cite[Example 3.6]{salamon}. 
\end{Remark}

\begin{Remark}
In Definition \ref{def:blowup} we have restricted to case when $\Sigma$ has positive genus.  When $\Sigma = S^2$, 
Karshon and Kessler
\cite{blowups} have defined a {\bf blowup form} on $(M_{S^2})_{k}$:
it is a symplectic form for which there exist pairwise disjoint 
embedded symplectic spheres in $L, \, E, \, E_1,\ldots,E_k$.
\end{Remark}

\begin{Definition}\labell{defequiv}
We say that  two symplectic forms $\omega_1$ and $\omega_2$ on $M$ are {\bf equivalent} if there exists a diffeomorphism $f$ of $M$ that acts trivially on the homology $H_{2}(M)$ and such that $f^{*}\omega_{2}$ can be connected to $\omega_{1}$ through a continuous path of symplectic forms.
\end{Definition}

The following lemma follows from the work of Gromov \cite{gromovcurves}, McDuff \cite{structure}, Lalonde-McDuff \cite[Theorem 2.4]{ML} and McDuff-Salamon \cite[Proposition 7.21]{intro}.
We give the proof in the positive genus case in \cite{HK:circle};
 the proof in the genus $0$ case is in \cite{algorithm}.

\begin{Lemma} \labell{lemequiv}
The set of blowup forms on $M$ is an equivalence class of symplectic forms.
\end{Lemma}

The set
$\calE(M)$
of exceptional classes in $H_2(M)$ is the same for all the blowup forms, as discussed in Lemma~\ref{calE and J}. 
An {\bf exceptional sphere} in a symplectic four-manifold $(M,\omega)$ is an embedded 
$\omega$-symplectic sphere of self intersection $-1$. 
A homology class $E \in H_2(M)$ is {\bf exceptional} if it is represented by an exceptional sphere.
For blowups of $S^2$-bundles over a Riemann surface $\Sigma$ with a positive genus, there is  an explicit identification of the classes in $\calE$, by 
a holomorphic argument of 
Biran \cite[Corollary 5.C]{biran:packing}. 
\begin{equation}\labell{Biran}
\calE(M)=\{E_1,\ldots,E_k, F-E_1,\ldots,F-E_k\}.
\end{equation}

Pinsonnault \cite[Proposition 3.13]{pinso} showed that every Hamiltonian circle action 
on a $k$-fold blowup of a ruled symplectic four-manifold can be obtained from an action 
on some ruled symplectic manifold $M_0$ by a sequence of $k$ equivariant blowups of some sizes. 
For blowups of irrational ruled symplectic manifolds, we use the identification of $\calE(M)$ in 
order to derive, from the same holomorphic tools used in the non-constructive result of 
Pinsonnault, a constructive characterization, in which we identify the ruled manifold $M_0$ 
and the sizes of the equivariant blowups.

\begin{Theorem} \labell{conjcount}
Let $k\geq 1$ be an integer, and
let  $M=(\Sigma \times S^2)_{k}$ {\color{red}$($respectively $M=(M_{\Sigma})_{k})$}.
Assume that the genus of $\Sigma$ is positive.
Let $\omega$ be a blowup form on $M$ whose cohomology class is encoded by a vector 
$(\lambda_F,\lambda_B;\delta_1,\ldots,\delta_k)$; if $k \geq 2$, assume that  $[\omega]$ is in 
$g$-reduced form.
Then every Hamiltonian circle action is equivalent to 
\begin{itemize}
\item[\textcircled{a}] one that is obtained from a Hamiltonian circle action on 
$\Sigma \times S^2$  {\color{red}$($resp.\ $M_{\Sigma})$} with the symplectic 
form whose cohomology class is encoded by $(\lambda_F,\lambda_B)$ by a 
sequence of $S^1$-equivariant symplectic blowups of sizes
$\delta_1$,$\delta_2$, $\delta_3$, $\ldots$, $\delta_k$, in this order,
\end{itemize} and to 
\begin{itemize}
\item[\textcircled{b}] one that is obtained from a Hamiltonian circle action on  
$M_{\Sigma}$ {\color{red}$($resp.\ $\Sigma \times S^2)$}  with the symplectic form 
whose cohomology class is encoded by $(\lambda_F,  \lambda_B+\half \lambda_F-\delta_1)$  
by a sequence of $S^1$-equivariant symplectic blowups of sizes
 $\lambda_F-\delta_1$, $\delta_2$, $\ldots$, $\delta_k$, in this order.
\end{itemize}
\end{Theorem}

\noindent Moreover, Theorem~\ref{conjcount} holds if we replace 
``Hamiltonian circle action'' with ``action of a compact Lie group $G$ that preserves 
$\omega$ and induces the identity morphism on $H_2(M)$.'' 
To complete the characterization 
we  provide a description of 
the Hamiltonian circle actions on $\Sigma \times S^2$ and on $M_{\Sigma}$ in Proposition~\ref{zerocase}. 

\begin{proof}[Proof of Theorem \ref{conjcount}.]
Consider an action on $(M,\omega)$ of a compact Lie group $G$ 
that preserves $\omega$ and induces  the identity morphism on $H_2(M)$.  
Let $J_G$ be a $G$-invariant $\omega$-tamed almost complex structure on $M$.
The structure $J_G$ is the almost complex structure associated by the polar decomposition to the invariant Riemann metric $g$ obtained by averaging some Riemann metric $g'$ along the action of $G$ with respect to the Haar measure: $g(u,v):=\int_{G}g'({\sigma_{a}}_{*}u,{\sigma_{a}}_{*}v)d a$.

 First assume $k \geq 2$.
 By the identification  \eqref{Biran} of $\calE(M)$ and the $g$-reduced assumption, the class $E_k$ is of minimal symplectic area in $\calE(M)$. 
 Hence, by Lemma \cite[Lemma 1.2]{pinso}, $E_k$ is represented by an embedded $J_G$-holomorphic sphere $C$.
To see that $C$ is $G$-invariant, apply the positivity of  
intersections  \cite[Proposition 2.4.4]{nsmall}, the fact that $E \cdot E=-1$, and the 
assumption that the action is trivial on $H_2(M)$ to get that for every $a \in G$, the 
$J_G$-holomorphic sphere $aC$ equals the $J_G$-holomorphic sphere $C$.

 Equivariantly blowing down along this sphere yields a $G$-action on $(\Sigma \times S^2)_{k-1}$ (or $(M_{\Sigma})_{k-1}$) with the blow up form whose cohomology class is encoded by $({\lambda_F,\lambda_B;\delta_1,\ldots,\delta_{k-1}})$ and is in $g$-reduced form, as assured by Lemma \ref{blowdown Ek} (on blow downs). 

We now let $k \geq 1$. By $k-1$ repeated  applications of this blow down, we reduce  the theorem to the following claim:
 There exist embedded $G$-invariant symplectic spheres $C$ in $E_1$ and $D$ in 
 $F-E_{1}$ in $(\Sigma \times S^2)_{1}$  (resp.\ $(M_{\Sigma})_{1}$), with the blowup 
 form encoded by $(\lambda_F,\lambda_B;\delta_1)$, and blowing down equivariantly 
 along $C$ yields a $G$-action on $\Sigma \times S^2$ (resp.\ $M_{\Sigma}$) with the 
 symplectic form encoded by $({\lambda_F,\lambda_B})$, and blowing down along $D$ 
 yields a $G$-action on $M_{\Sigma}$ (resp.\ $\Sigma \times S^2$) with the symplectic 
 form encoded by $({\lambda_F,\lambda_B+\half \lambda_F-\delta_1})$.
 This claim follows from the existence of embedded $J_G$-holomorphic spheres in 
 $E_1$ and in $F-E_1$, stated in Lemma \ref{martin}, which is the positive genus version 
 of the genus zero \cite[Lemma~2.2]{pinso2}, and from Lemma \ref{blowdown Ek}  and 
 Corollary \ref{corbdowne12 prime} (on blow downs).
\end{proof}

Our second main result provides a unique representation for a blowup form on the $k$-fold
blowup of a ruled symplectic four-manifold $\Sigma \times S^2$ or $M_{\Sigma}$.

\begin{Theorem} \labell{thmexu}
Let $k \geq 2$ be an integer, and let  $M=(\Sigma \times S^2)_{k}$ or $M=(M_{\Sigma})_{k}$. 
Assume that the genus $g(\Sigma)$ is positive. 
Given a blowup form $\omega$ on $M$, there exists a unique 
blowup form $\omega'$ whose cohomology class is in $g$-reduced 
form such that $(M, \omega) \cong (M,\omega')$.
\end{Theorem}

\noindent 
In fact, we prove even more than is stated in Theorem~\ref{thmexu}.  
We include an algorithm, Algorithm~\ref{alg:cremona}, which puts
a blowup form into its unique $g$-reduced form, by a composition of permutations of the 
$\delta_i$s and a map we denote the {\bf Cremona transformation} 
(see Definition \ref{Cremdef}). The latter map is a positive genus version of the 
(usual) Cremona transformation on blowups of $M_{S^2}$; it is motivated 
by the combinatorial description of the transformation as interwining 
two ways to obtain the same Hamiltonian $S^1$-manifold, as shown in
Figure~\ref{fig:cremona demo}.

\begin{figure}[h] 
\begin{tabular}{@{}c@{}}{\includegraphics[height=2.25cm]{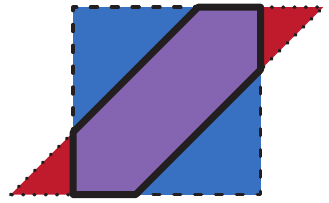}}\end{tabular}
 \hskip 1cm
\begin{tabular}{@{}c@{}}{\includegraphics[height=2.5cm]{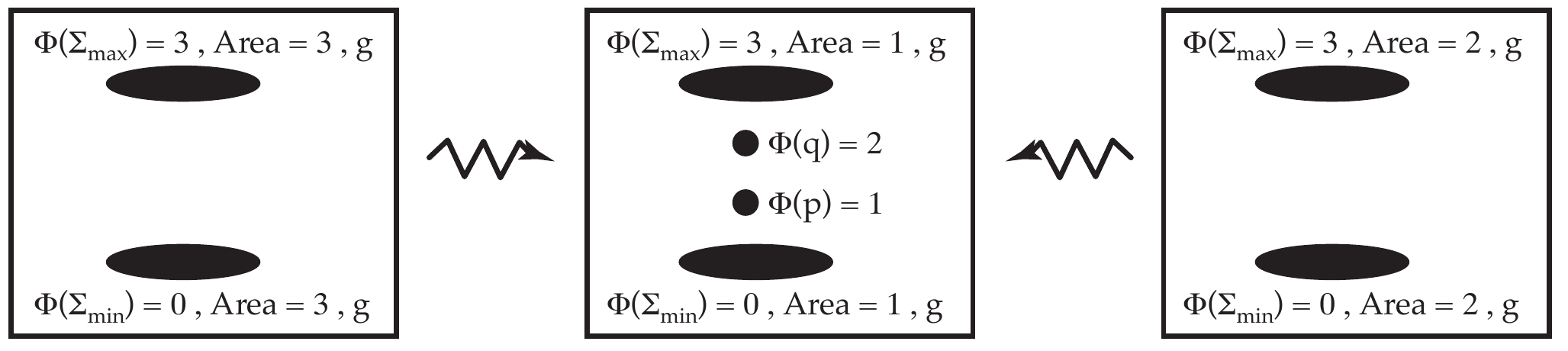}}\end{tabular}

\caption{{\bf On the left}, the same toric manifold, with the purple hexagonal
moment image, can be seen in two ways as a $2$-fold blowup of $\CP^1\times \CP^1$, 
with different symplectic forms and blowup sizes.  
In this case, the Cremona transformation turns the
vector $(3,3;2,2)$ (the blowup of the blue square)
 into $(3,2;1,1)$ (the blowup of the red parallelogram). \newline
{\bf On the right} we can see the $S^1$-manifold represented by the middle decorated
graph as a blowup from the left $(3,3)$ with size $2$ blowups, one each from the top and 
bottom; or from the right $(3,2)$ with size $1$ blowups, one each from the top and the bottom. 
As before,  the Cremona transformation turns the vector $(3,3;2,2)$ into 
$(3,2;1,1)$.}
\label{fig:cremona demo}

\end{figure}

Theorem~\ref{thmexu} also provides a way to determine when two
symplectic forms on the same manifold, on $(\Sigma \times S^2)_{k}$ or on $(M_{\Sigma})_{k}$, 
are diffeomorphic:
they must have the same $g$-reduced form.
In the case when 
$g(\Sigma)=0$ and $M=(M_{\Sigma})_{k}$, the existence and uniqueness of a $g$-reduced form is proved 
in \cite[Theorem~1.4]{blowups}.

\begin{Remark}\labell{variation of reduced}
Li-Li \cite{Li-Li02} and Li-Liu \cite{Li-Liu}, in the context of relating the symplectic genus of an integral cohomology class to its minimal genus, 
define reduced forms for blowups of $S^2 \times \Sigma$ (with $g(\Sigma)>0$)
by different conditions. 
First, they work with homology classes with $\Z$ coefficients whereas 
we work with cohomology classes over $\R$. 
More significantly, their conditions translate (by Poincar\'e\ duality) to the conditions on an 
integral $2$-cohomology class:
$$\lambda_B \geq \delta_1 \geq \ldots \geq \delta_k \geq 0.$$ 
By contrast, we require 
$\lambda_F\geq \delta_1+\delta_2$, and neither require  
$\lambda_B \geq \delta_1$ nor positivity. Note that the vector $(3,3;2,2)$ from
Figure~\ref{fig:cremona demo} is reduced in the sense of \cite{Li-Li02,Li-Liu}.  The
Cremona transformation turns it into $(3,2;1,1)$, which encodes a blowup
form diffeomorphic to $(3,3;2,2)$ that is $g$-reduced.
The requirement $\lambda_F \geq \delta_1+\delta_2$ is essential for the key 
claim that the class $E_k$ is of minimal symplectic area in our proof of 
Theorem \ref{conjcount}.  
 
Li and Li \cite[Lemma 3.4]{Li-Li02} prove that any class in $H_{2}$ over $\Z$ of 
nonnegative square is equivalent to a reduced  class under the 
action of orientation-preserving diffeomorphisms; they provide an algorithm to 
find the reduced class using reflections along integral homology classes 
that are realized by orientation-preserving diffeomorphisms.  In Li-Li's terminology, 
the Cremona transformation and the $(\delta_i,\delta_j)$-transpositions we use in 
Algorithm~\ref{alg:cremona} correspond to reflections along $F-E_1-E_2$ and along 
$E_i-E_j$, respectively.  However Li and Li's algorithm uses other transformations 
which we do not allow, e.g., $-\mathbb{1}$ and the reflections along 
$r_iF+\varepsilon_iE_i$.  Their reflections represent diffoemorphisms, but possibly
not symplectomorphisms. Their algorithm might send a vector encoding (the dual of) 
a blowup form $\omega$ to a vector that cannot encode a blowup form, 
or to a vector encoding a blowup form $\omega'$ that is not diffeomorphic 
to $\omega$. For example, Li and Li's algorithm sends the ($g$-reduced) vector
$(6,1;2,1)$ by reflection over $F-E_1$ to $(4,1;1,0)$, which does not represent a blowup
form.  Their algorithm sends the ($g$-reduced) vector $(12,2;3,3)$ by reflections over
$F-E_1$ and over $F-E_2$ to $(10,2;1,1)$.  Those vectors represent manifestly
different symplectic forms: the former admits no Hamiltonian circle actions, while
the latter does admit a Hamiltonian circle action. 
\end{Remark}

We conclude this section with a brief outline of the rest of the paper.
In Section~\ref{counting}, we recall the notion of the decorated graph associated to
a Hamiltonian circle action, and describe the effect of an equivariant symplectic blowup 
on the graph. We then derive from Theorem \ref{conjcount} results on the existence and number of Hamiltonian
circle actions on a fixed symplectic four-manifold.
Section~\ref{se:blowups} examines symplectic blowups of ruled symplectic four-manifolds and 
highlights the symplectic facts and techniques we need to prove the main theorems.
We deduce from Li-Liu's characterization of symplectic forms \cite{Li-Liu} a necessary
and sufficient condition for a vector to encode a cohomology class of a blowup form,
and describe how these techniques can compute certain symplectic invariants like
Gromov width and packing number (see Remarks~\ref{rem:G width} and \ref{rem:packing}).
In Section~\ref{se: at most 2 blowups}, we use the correspondence between Hamiltonian circle actions and decorated graphs to deduce results on symplectic manifolds from combinatorial observations.
Finally, we prove the existence portion of Theorem~\ref{thmexu} in Section \ref{sec:reduce}, 
and the uniqueness in Section \ref{sec:minimal set}.

%%%%%%%%%%%%%%%%%%%%%%%%%%%%
\section{Counting Hamiltonian circle actions on blowups\\ of irrational ruled manifolds}\label{counting}
%%%%%%%%%%%%%%%%%%%%%%%%%%%%

\subsection*{Decorated graphs} If a circle action $S^1\acts (M,\omega)$ is Hamiltonian, then there is a real-valued momentum map
$\Phi: M\to \R$.  This is a Morse-Bott function with
critical set corresponding to the fixed points.
When $\dim(M)=4$, 
the critical set can only consist
of isolated points and two-dimensional submanifolds.  The latter
can only occur at the extrema of $\Phi$.  
To $(M,\omega,\Phi)$ Karshon associates the following {\bf decorated graph} \cite[\S 2.1]{karshon}.
For each isolated fixed point $p$ there is a vertex $\langle p \rangle$,
labeled by the real number $\Phi(p)$.
For each two dimensional component $S$  of the fixed point set there is
a {\bf fat} vertex $\langle S \rangle$ labeled by two real numbers 
and one integer: the {\bf momentum map label} $\Phi(S)$, 
the {\bf area label} $\frac{1}{2\pi} \int_S \omega$, 
and the {\bf genus} $g$ of the surface $S$. 
A {\bf $\bfZ_k$-sphere} is a gradient sphere in $M$ on which $S^1$ acts with isotropy $\bfZ_k$.
For each $\bfZ_k$-sphere containing two fixed points $p$ and $q$,
the graph has an edge connecting the vertices $\langle p \rangle$
and $\langle q \rangle$ labeled by the integer $k$; the size $\frac{1}{2\pi} \int_S \omega$ of a $\bfZ_k$-sphere $S$ 
 is $1/k$ of the difference of the moment map values of its vertices.      
We note that vertical translations of the graph correspond to equivariant symplectomorphisms, and flips correspond to automorphisms of the circle.

\subsection*{Hamiltonian circle actions on irrational ruled symplectic manifolds}
Karshon's results  \cite[Theorem 6.3 and Lemma~6.15]{karshon} and  \cite[Table in Proof of Lemma~1]{YK:max_tori}
imply that
Hamiltonian circle actions on a ruled symplectic four-manifolds over a closed Riemann surface of positive genus $g$ 
correspond to decorated graphs that consist of two fat vertices with the same genus label $g$. Moreover, Karshon's results yield the following characterization of Hamiltonian circle actions on irrational ruled symplectic four-manifolds. See also \cite[Corollary 3.12]{pinso}.

\begin{figure}[h] 
\centering{
\includegraphics[width=10cm]{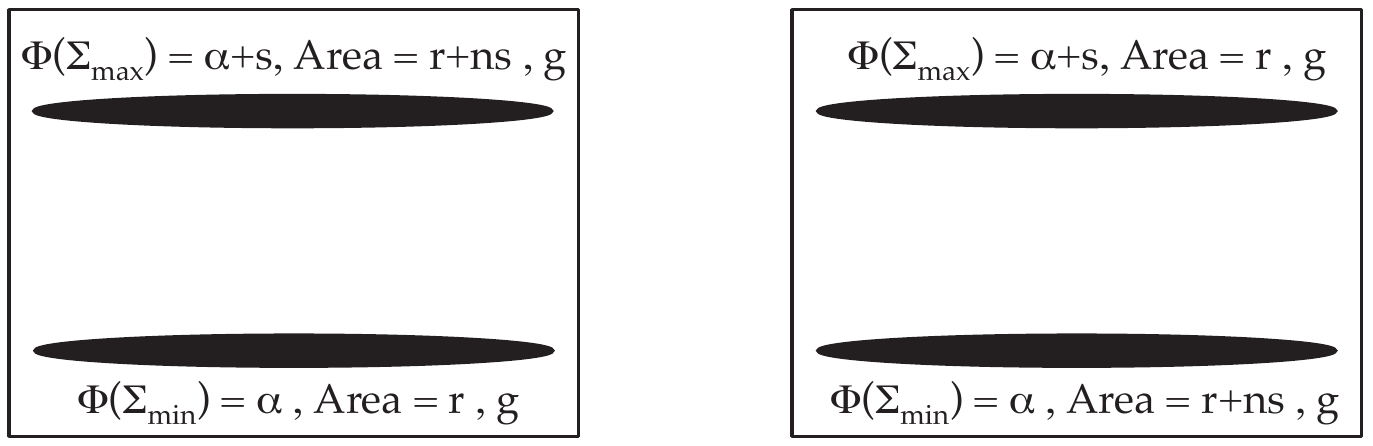}
\caption{Decorated graphs with exactly two fat vertices, differing by a flip. } 
\label{fig:fat vertices}
}
\end{figure}

\begin{Proposition}\labell{zerocase}
Let $(M,\omega)$ be a ruled symplectic manifold over a Riemann surface $\Sigma$ of positive genus $g$.
Up to equivariant symplectomorphisms
and automorphisms of the circle, the number of Hamiltonian circle actions on $(M,\omega)$  is $\left\lceil \frac{\lambda_B}{\lambda_F} \right\rceil$ if $M=\Sigma \times S^2$ and $\left\lceil \frac{\lambda_B-\half \lambda_F}{\lambda_F} \right\rceil$ if $M=M_{\Sigma}$.

Moreover, there is a one-to-one correspondence
between inequivalent Hamiltonian $S^1$-actions on $(M,\omega)$ and the nonnegative integers $0 \leq n<2\frac{\lambda_B}{\lambda_F}$ 
such that
\begin{itemize}
\item the integer $n$ is even if $\pi \colon M \to \Sigma$ is the trivial bundle; $n$ is odd if $\pi \colon M \to \Sigma$ is the non-trivial bundle.
\item there exists $r \in \R^{>0}$ such that  $\half(r+(r+n\lambda_F))=\lambda_B$.
\end{itemize}
\end{Proposition}

\begin{Remark}
Note that the enumeration of the inequivalent circle actions on a ruled symplectic manifold over a Riemann surface $\Sigma$ of positive genus $g$ is analogous to the enumeration of the inequivalent toric actions on a ruled symplectic manifold over $S^2$ as in  \cite{YK:max_tori}.
\end{Remark}

\subsection*{The effect of an $S^1$-equivariant blowup on the decorated graph}
Recall that we can think of a symplectic blowup of size $\delta=r^2/2$ as cutting out an embedded ball of radius $r$ and identifying the boundary to an exceptional sphere via the Hopf map. This carries a symplectic form that integrates on the sphere to $2\pi \delta$.
For more details see \cite{GS:birational} and \cite[Section 7.1]{intro}. If the embedding of the ball is $G$-equivariant centered at a $G$-fixed point, then the $G$-action extends to the symplectic blowup, see details in \cite{kk}. If the action is Hamiltonian, its moment map naturally extends to the equivariant symplectic blowup.

In Figure~\ref{fig:eq circle-blowup}, we describe all the possible effects of an $S^1$-equivariant symplectic blowup of a ruled 
manifold over a surface of positive genus, up to flips.  If the surface is of genus $0$, an $S^1$-equivariant symplectic blowup can
have additional effects. For a complete description of Hamiltonian $S^1$-blowups of $4$-manifolds,
see \cite[\S 6]{karshon}. We observe that flipping commutes with $S^1$-equivariant symplectic blowup.

\begin{figure}[h] 
\centering{
\begin{subfigure}[b]{0.58\textwidth}
\raisebox{7mm}{\includegraphics[width=8cm]{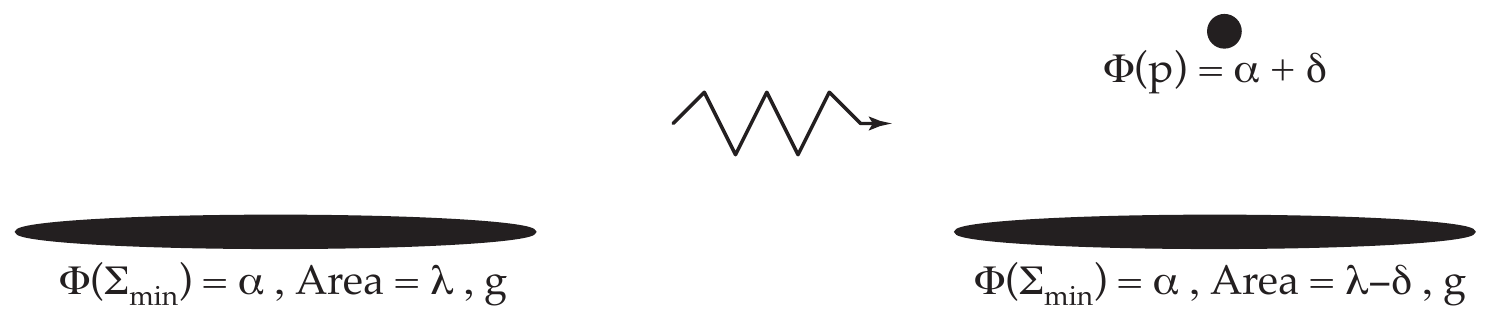}}
\end{subfigure} 
\begin{subfigure}[b]{0.38\textwidth}
\includegraphics[width=6cm]{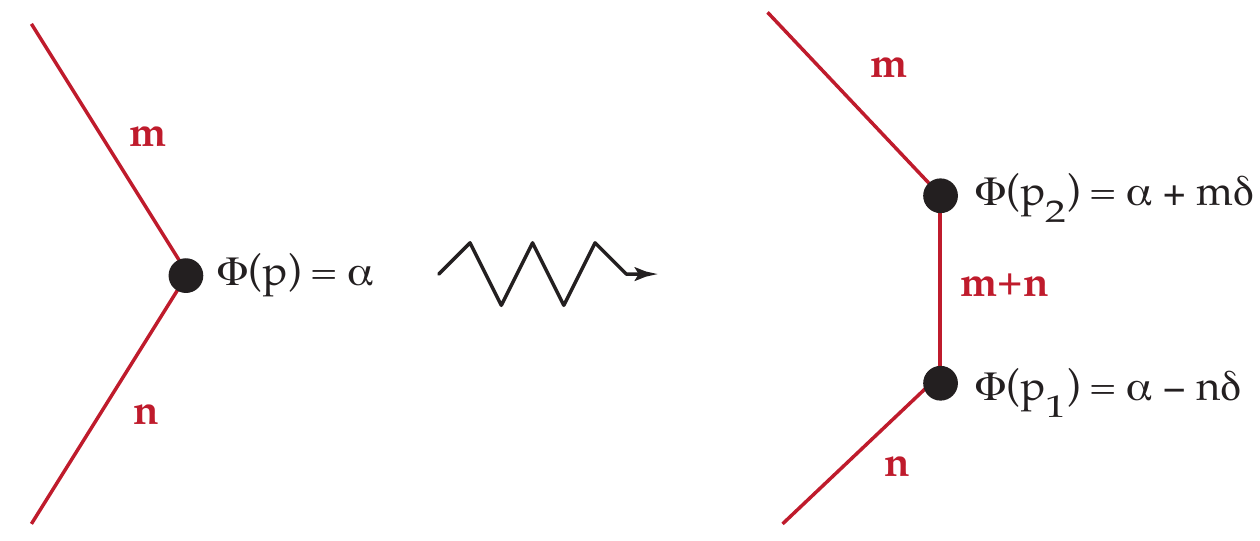}
\end{subfigure}
\caption{The effect on the decorated graph of an  $S^1$-equivariant blowup of size $\delta$: at a minimum surface on the left, and at an interior isolated fixed point on the right.}
\label{fig:eq circle-blowup}
}
\end{figure}

\begin{Remark}\labell{bupconditions}
By  \cite[Proposition 7.2]{karshon},
a Hamiltonian $S^1$-space
admits an $S^1$-equivariant blowup of size $\delta > 0$ centered at
some fixed point
if and only if one obtains a \emph{valid} decorated graph after the blowup. That is, the 
(fat or not) vertices created in the blowup do not surpass the other pre-existing (fat or not) 
vertices in the same chain of edges, and the fat vertices after the blowup have positive size labels. 
For quantitative description, see \cite[Lemma 3.3]{kk}.
\end{Remark}

\subsection*{An algorithm to count the (non-equivalent) Hamiltonian circle actions on a $k$-fold blowup of an irrational ruled manifold}

Let $k\geq 1$. Let $\Sigma$ be a Riemann surface of positive genus. 
When the cohomology class of a blowup form $\omega=\omega_{\lambda_F,\lambda_B;\delta_1,\ldots,\delta_k}$ on $M=(\Sigma \times S^2)_{k}$ ($M=(M_{\Sigma})_{k}$) is $g$-reduced, Theorem \ref{conjcount} reduces the question of counting the number of Hamiltonian circle actions on $(M,\omega)$ to a combinatorial one. It is enough to count the number of possible ways to get valid decorated graphs by $k$ $S^1$-blowups of sizes $\delta_1,\ldots,\delta_k$ in descending order, starting from  a decorated graph with two fat vertices, the top of size $\lambda_B-\frac{n}{2}\lambda_F$ and the bottom of size $\lambda_B+\frac{n}{2}\lambda_F$, for $0 \leq n < 2\frac{\lambda_B}{\lambda_F}$ that is even if $M=(\Sigma \times S^2)_{k}$ and odd if $M=(M_{\Sigma})_{k}$. We count the resulting decorated graphs up to vertical translations and flips.

By Corollary \ref{cor-same}, it is enough to count the Hamiltonian circle actions on symplectic $k$-fold blowups of $\Sigma \times S^2$.
The counting algorithm was implemented by Tair Pnini. 
Let $\omega_{\lambda_F,\lambda_B;\delta_1,\ldots,\delta_k}$ be a blowup form on $(\Sigma \times S^2)_{k}$ 
whose cohomology class is in $g$-reduced form and is encoded by a vector 
$(\lambda_F,\lambda_B;\delta_1,\ldots,\delta_k)$.
At step $0$ the program creates the set of graphs with two fat vertices determined by $\lambda_B, \lambda_F$, 
corresponding to the even integers $0 \leq n < 2 \frac{\lambda_B}{\lambda_F}$.
At step $i$, for $1 \leq i \leq k$, it  
creates the set of graphs (up to equivalence) which may be obtained by performing a valid blowup of size $\delta_i$ 
on a graph in the set created at step $i-1$. 
The output of the program is the number of graphs in the set obtained at step $k$.

This is a ``greedy" algorithm. We use data structures designed to reduce the number of tests for equivalence 
between pairs of graphs and to optimize the equivalence test. A detailed description of the data structures and the 
pseudo-code of the algorithm is given in Appendix~\ref{tair} by Tair Pnini.

\begin{Corollary}
Let $k\geq 1$ and assume that we are performing $k$ blowups of equal size: 
\newline $\delta_1=\ldots=\delta_k=\varepsilon$.
\begin{itemize}
\item Assume $2 \varepsilon < \lambda_F$. The number of Hamiltonian circle actions on $((\Sigma \times S^2)_{k},\omega_{\lambda_F,\lambda_B;\varepsilon,\ldots,\varepsilon})$ equals 
\begin{equation}\label{eq:2e<lambda_F}
\left(\sum_{n=1}^{\left\lceil{\frac{\lambda_B}{\lambda_F}} \right \rceil-1}
\left( \ \sum_{j=0}^{k}\delta_{j \varepsilon < (\lambda_B-{n} \lambda_F)} \delta_{(k-j)\varepsilon<(\lambda_B+{n}\lambda_F)}\right)\right)
+
\sum_{j=0}^{\left \lfloor \frac{k}{2} \right \rfloor}\delta_{j \varepsilon < \lambda_B} \delta_{(k-j)\varepsilon<\lambda_B},
\end{equation}
where $$\delta_{a<b}=\begin{cases}  1\text{ if }a<b\\ 0 \text{ if }a \geq b\end{cases}.$$ 
We count the possible ways to preform $k$ blowups of size $\varepsilon$, each at a fat vertex, starting from one of the decorated graphs of two fat vertices listed in Proposition \ref{zerocase}.
\item
Assume $2 \varepsilon =\lambda_F$.
The number of Hamiltonian circle actions  on $((\Sigma \times S^2)_{k},\omega_{\lambda_F,\lambda_B;\varepsilon,\ldots,\varepsilon})$ equals 
$$
\eqref{eq:2e<lambda_F} - \left(\sum_{n=1}^{\left \lceil{\frac{\lambda_B}{\lambda_F}}\right \rceil-1} 
\left(\ \sum_{j=0}^{k-2} {\delta_{j\varepsilon<(\lambda_B-n\lambda_F)}\cdot\delta_{(k-2-j)\varepsilon<(\lambda_B+(n-1)\lambda_F)}}\right)\right).
$$
We have subtracted the equivalent actions, like those in Figure~\ref{fig:equal size same}
where blowups of size $\frac{\lambda_F}{2}$ from the top or bottom fat vertices will end up at the same height, resulting in
the same decorated graphs.
\item 
Similarly, if $2 \varepsilon < \lambda_F$, the number of Hamiltonian circle actions on  $((M_{\Sigma})_{k},\omega_{\lambda_F,\lambda_B;\varepsilon,\ldots,\varepsilon})$ equals the sum
\begin{equation}\label{eq:2e<<lambda_F}
\sum_{n=0}^{\left \lceil{\frac{\lambda_B-\half \lambda_F}{\lambda_F}}\right \rceil-1}
\left(\sum_{j=0}^{k}\delta_{j \varepsilon < \left(\lambda_B-\frac{2n+1}{2} \lambda_F\right)} \delta_{(k-j)\varepsilon<\left(\lambda_B+\frac{2n+1}{2}\lambda_F\right)}\right).
\end{equation}
If $2 \varepsilon =\lambda_F$, it equals 
$$
\eqref{eq:2e<<lambda_F} - \left(\sum_{n=1}^{\left \lceil{\frac{\lambda_B-\half \lambda_F}{\lambda_F}} \right \rceil-1} 
\left(\sum_{j=0}^{k-2} {\delta_{j\varepsilon<\left(\lambda_B-\frac{2n+1}{2}\lambda_F\right)}\cdot\delta_{(k-2-j)\varepsilon<\left(\lambda_B+\frac{2(n-1)+1}{2}\lambda_F\right)}}\right)\right).
$$
\item In particular, if  $k \varepsilon \geq 2 \lambda_B$ then there are no Hamiltonian circle actions.
\end{itemize}
\end{Corollary}

\begin{figure}[h] 
\centering{
\includegraphics[width=15.5cm]{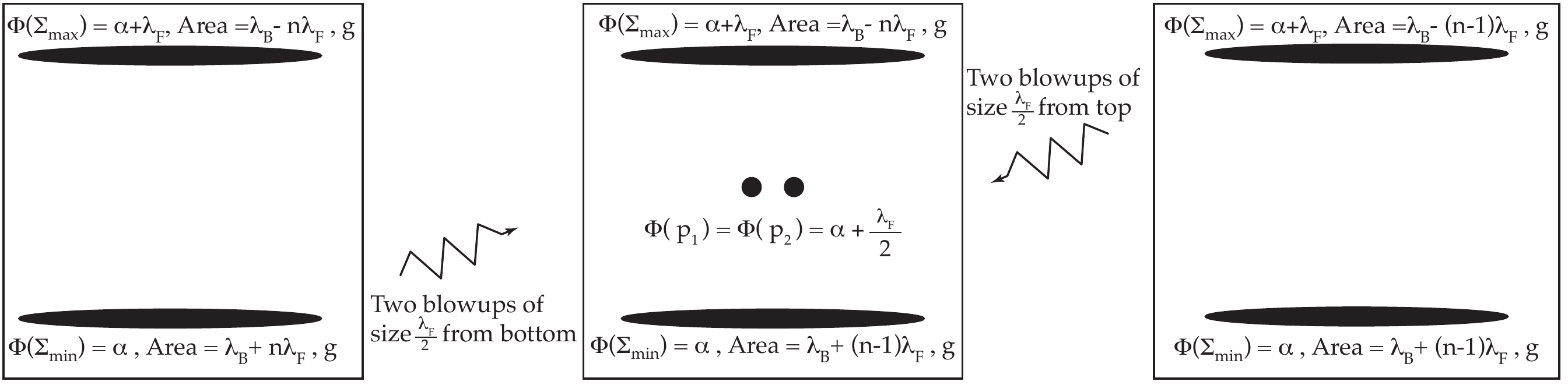}
\caption{
This figure shows coincidentally equivalent decorated graphs arising from blowups of different graphs.
} 
\label{fig:equal size same}
}
\end{figure}

For example, for $k \geq 2$ there are no Hamiltonian circle actions on $(\Sigma \times S^2)_{k}$ 
with a blowup form $\omega_{2,1;\frac{2}{k},\ldots \frac{2}{k}}$; by Lemma \ref{poslemma}, such a 
blowup form exists.  This lack of Hamiltonian actions also appears for equal blowups in the simply connected case; 
see for example \cite{kk}.

\begin{Corollary}
Let $k \geq 1$. The number of Hamiltonian circle actions on $((\Sigma \times S^2)_{k},\omega_{\lambda_F,\lambda_B;\delta_1,\ldots,\delta_k})$ is at most 
\begin{equation} \labell{eqmax}
 \left(\left\lceil \frac{\lambda_B}{\lambda_F} \right\rceil-\half \right) (k+1)!
\end{equation}
where $\lceil a \rceil$ denotes the smallest integer greater than or equal to $a$.
If 
\begin{itemize}
\item $\sum_{i=1}^{k}\delta_i < \lambda_F$, and
\item $\sum_{i=1}^{k}\delta_i < \lambda_B-\frac{n}{2}\lambda_F$ for every even $n < 2\frac{\lambda_{B}}{\lambda_{F}}$, 
and
\item for every $j$, the sum $\sum_{i=j+1}^{k}{\delta_i}<\delta_j$, and moreover
\item for every $j$ and $s$, we have $\sum_{i=1}^{s}F_{i+1} \delta_{j+i}< F_{1}\delta_{j}$, where $\{F_i\}$ is the sequence of Fibonacci numbers,
\end{itemize}
then the number of Hamiltonian circle actions on $((\Sigma \times S^2)_{k},\omega_{\lambda_F,\lambda_B;\delta_1,\ldots,\delta_k})$ equals \eqref{eqmax}.
\end{Corollary}

For example, if $\lambda_B \in \N$, $\lambda_F=1$ and $\delta_{i}=\frac{1}{4^i}$ for $1 \leq i \leq k$, the number of Hamiltonian circle actions is $(\lambda_B -\half) \cdot (k+1)!$. (By Lemma \ref{poslemma},
 there exists a blowup form on $(\Sigma \times S^2)_{k}$ whose cohomology class is encoded by this vector, for every $k \geq 1$.)

%%%%%%%%%%%%%%%%%%%%%%%%%%%%%%%%%%
\section{Blowups of ruled symplectic four-manifolds}  \label{se:blowups}
%%%%%%%%%%%%%%%%%%%%%%%%%%%%%%%%%%

Let $\Sigma$ be a compact connected Riemann surface. Fix an integer $k \geq 0$.
Let $M=(\Sigma \times S^2)_{k}$ or $M=(M_{\Sigma})_{k}$.
As a smooth manifold, $M$ is obtained by $k$ iterations of a connected sum with $\ol{\CP^2}$, 
$\CP^2$ equipped with the opposite orientation. In general, the connected sum
construction $\tN=N\# \ol{\CP^2}$ 
depends on a number of choices,
but if $\tN'$ is another connected sum 
obtained from different choices,
then there exists a diffeomorphism
from $\tN$ to $\tN'$ that respects the splitting $H_2(\tN) = H_2(N) \oplus \Z E$ induced by the constructions.
This follows from standard arguments in differential topology (see e.g., \cite[Chapter~10]{Brocker-Janich}). 
In particular, if $M$ is obtained from $\Sigma \times S^2$ or $M_\Sigma$ by blowing up at $k$ distinct points 
and $M'$ is obtained by blowing up at a permutation of the same points, 
then the two manifolds are the same,  up to relabeling the exceptional classes $E_1,\ldots,E_k$. 
\begin{Lemma} \labell{permute diff top}
Fix a vector $(\lambda_F,\lambda_B;\delta_1,\ldots,\delta_k)$ and a permutation $\sigma\in S_k$.
Suppose that there exists a blowup form $\omega$ on $M$
whose cohomology class is encoded by this vector.
Then there exists a blowup form $\omega'$ 
whose cohomology class is encoded by the vector
$(\lambda_F,\lambda_B;\delta_{\sigma(1)},\ldots,\delta_{\sigma(k)})$
and $(M,\omega)$, and $(M,\omega')$ are symplectomorphic.
\end{Lemma}

\begin{noTitle}
For any manifold $M$,
an {\bf almost complex structure}  is an automorphism $J \colon TM \to TM$
such that $J^2 = -\mathbb{1}$. 
An almost complex structure $J$ on $M$ is {\bf tamed}
by  a symplectic form $\omega$
if $\omega(u,Ju)>0$ 
for all nonzero tangent vectors $u \in TM$.
Let $ \calJ_\tau(M,\omega)$ denote
the set of almost complex structures $J$ that are tamed by $\omega$.
The space $\calJ_\tau(M,\omega)$
is a nonempty contractible open subset of the space of all almost complex structures
\cite[Proposition 2.51]{intro}.

A \textbf{(parametrized) $J$-holomorphic sphere} is a map $f \colon (\CP^1,j) \to (M,J)$  that satisfies
the Cauchy-Riemann equations 
$
df \circ j = J \circ df
$
at every $p \in \CP^1$.
An embedding is a one-to-one immersion which is a homeomorphism 
with its image.
An \textbf{embedded $J$-holomorphic sphere} $C \subset M$ is the image 
of a $J$-holomorphic embedding $f \colon \CP^1 \to M$.
If $J$ is $\omega$-tamed then such a $C$ is an embedded $\omega$-symplectic sphere.
We will refer to the genus zero Gromov Witten invariant (with point constraints) 
$ \GW \colon H_2(M) \to \Z .$
For the precise definition, see~\cite{nsmall}.
Fixing a symplectic form $\omega$, if $\GW(A) \neq 0$, then for generic $\omega$-tamed almost complex structure $J$ there exists a $J$-holomorphic sphere in the class $A$.

Because the first Chern class of a complex vector bundle does not change
under a continuous deformation of the fibrewise complex structure,
the first Chern class $c = c_1(TM,J)$ is the same
for all $J \in \calJ_\tau(M,\omega)$.
It follows from Lemma \ref{lemequiv} that this first Chern class and the Gromov-Witten invariant are the same for all the blowup forms on $M=(\Sigma \times S^2)_{k}$ or $M=(M_{\Sigma})_{k}$. We denote the first Chern class and the Gromov Witten invariant associated to any blowup form on $M$ by $c_{1}(TM)$ and $\GW$.
\end{noTitle}

\begin{Facts} \labell{h2facts}
We will make use of the following facts.  
\begin{enumerate}
\item The classes $F$, $B$, and $E_1,\ldots,E_k$ form a basis of  the homology group 
$H_2((\Sigma \times S^2)_{k})$. 
The classes $F$, $B_{-1}$ ($B_1$), and  $E_1,\ldots,E_k$ form a basis of  $H_2((M_{\Sigma})_{k})$.
Recall that we have classes $B_n= [\sigma_n(\Sigma)]$ 
in  $H_2((\Sigma \times S^2)_{k})$ when $n$ is even
and in $H_2((M_{\Sigma})_{k})$ when $n$ is odd.

\item 
The intersection numbers are given in the following table, where $i\neq j$.

\renewcommand{\arraystretch}{1.2}
\begin{equation}\label{eq:intersection}
\begin{array}{l|r|r|r|r|r|}
& F & B & B_n & E_i & E_j\\ \hline
F & 0 & 1& 1& 0& 0\\ \hline
B & 1 & 0 &\frac{n}{2} & 0& 0\\ \hline
B_n& 1& \frac{n}{2} & n & 0 & 0\\ \hline
E_i & 0 & 0&0& -1&0\\ \hline
E_j & 0&0&0&0&-1\\ \hline
\end{array}
\end{equation}

\noindent Note also that $B_n\cdot B_{-n} = 0$.

\item For the first Chern class $c_{1}(TM)$ of a blowup form on $M$, $M=(\Sigma \times S^2)_{k}$ or $M=(M_{\Sigma})_{k}$, 
$$c_{1}(TM)(F)=2, \,\,\,  c_{1}(TM)(E_{i})=1\,\, \forall 1\leq i \leq k.$$
In $M=(\Sigma \times S^2)_{k}$,
$$c_{1}(TM)(B)=2-2g(\Sigma).$$
In $M=(M_{\Sigma})_{k}$,
$$c_{1}(TM)(B_{-1})=1-2g(\Sigma).$$

\item The genus zero Gromov-Witten invariant $\GW(F)$ with respect to any blowup form is not zero.
\end{enumerate}
 \end{Facts}
 The last item follows from  \cite[\S 4]{structure}
and the fact that  Gromov-Witten invariants are
consistent under the natural inclusions 
$$H_2((\Sigma \times S^2)_k) \to H_2((\Sigma \times S^2)_{k+1}) 
\mbox{ and }H_2((M_{\Sigma})_k) \to H_2((M_{\Sigma})_{k+1}),
$$
as in \cite[Theorem~1.4]{hu}, \cite[Proposition 3.5]{LP}, 
and the explanation in \cite[Appendix A]{algorithm}.

%%%%%%%%%%%%%%%%%%%%%%%%%%%%%%%%%%%%%
\subsection*{Exceptional classes}\label{se:exceptional}
%%%%%%%%%%%%%%%%%%%%%%%%%%%%%%%%%%%%%

We compile a number of results about exceptional classes on $k$-fold blowups of 
ruled symplectic four-manifolds.
The first lemma follows from McDuff's ``$C_1$ lemma'' \cite[Lemma 3.1]{structure}, 
Gromov's compactness theorem \cite[1.5.B]{gromovcurves}, the adjunction formula  \cite[Cor.~E.1.7]{nsmall},
and Lemma~\ref{lemequiv}. 
Full details are in \cite[Lemma~2.9]{algorithm}.

\begin{Lemma} \labell{calE and J}
Let  $M=(\Sigma \times S^2)_{k}$ ($M=(M_{\Sigma})_{k}$).
Let $E \in H_2(M)$ be a homology class.
Then the following are equivalent:
\begin{enumerate}
\item[(a)]
There exists a blowup form $\omega$ on $M$
such that the class $E$ is represented by an embedded $\omega$-symplectic
sphere with self intersection $-1$.
\item[(b)]
\begin{enumerate}
   \item[(i)]
   $c_1(TM)(E) = 1$;
   \item[(ii)]
$E \cdot E = -1$; and
   \item[(iii)]
the genus zero Gromov-Witten invariant $\GW(E) \neq 0$.
\end{enumerate}
\item[(c)]
For every blowup form $\omega'$ on $M$, 
the class $E$ is represented by an embedded $\omega'$-symplectic sphere 
with self intersection $-1$.
\end{enumerate}
\end{Lemma}

Lemma \ref{calE and J} guarantees that the set of exceptional classes is independent of the choice of a blowup form $\omega$ on ${M}$, and for every exceptional class $E$, we have $c_{1}(TM)(E)=1$ and $\GW(E) \neq 0$, i.e., for generic $\omega$-tamed almost complex structure $J$ there exists a $J$-holomorphic sphere in the class $E$.

The next lemma guarantees the existence of a $J$-holomorphic sphere in exceptional classes for \emph{every} $\omega$-tamed $J$ 
in the case of a single blowup of a ruled symplectic manifold over a Riemann surface
of positive genus. It follows from the adjunction formula \cite[Cor.~E.1.7]{nsmall}, the positivity of intersections \cite[Proposition 2.4.4]{nsmall}, and Gromov's compactness theorem  \cite[1.5.B]{gromovcurves}.
This is analogous to a result of Pinsonnault in the genus $0$ case, see
\cite[Lemma~2.2]{pinso2}, 
\cite[Lemma~3.24]{algorithm}.
We prove the positive genus case in Appendix~\ref{appendix1}.

\begin{Lemma} \labell{martin}
Let  $M=(\Sigma \times S^2)_{1}$ or $M=(M_{\Sigma})_{1}$.
Assume that $g=g(\Sigma)>0$.
Let $\omega$ be a blowup form on  ${M}$. 
Then for every  $J \in \calJ_\tau(M,\omega)$  there exists an embedded
$J$-holomorphic sphere in the class $E_1$ and  there exists an embedded
$J$-holomorphic sphere in the class $F-E_{1}$.
\end{Lemma}

The following lemma is a consequence of \eqref{Biran} and Facts \ref{h2facts} that we will need in the proof of the uniqueness theorem. 
\begin{Lemma}\labell{keepsF}
Let  $M=(\Sigma \times S^2)_{k}$ or $M=(M_{\Sigma})_{k}$.
Assume that $g(\Sigma)>0$.
\begin{itemize}
\item When $k=0$, the class $F$ is the unique class in $H_{2}(M;\Z)$ satisfying
\begin{enumerate}
\item Its symplectic area with respect to a symplectic ruling of the $S^2$-bundle is positive;
\item Its self intersection number is zero; and
\item Its coupling with the first Chern class $c_{1}(TM)$ equals two.
\end{enumerate}
\item When $k \geq 1$, the class $F$ is  the unique class in $H_{2}(M;\Z)$ satisfying
\begin{enumerate}
\item The intersection number of the class with every class in $\calE(M)$ is zero; and
\item Its coupling with the first Chern class $c_{1}(TM)$ equals two.
\end{enumerate}
\end{itemize}
\end{Lemma}

\begin{proof}
Denote $\hat{B}=B$ if  $M=(\Sigma \times S^2)_{k}$ and $\hat{B}=B_{1}$ if $M=(M_{\Sigma})_{k}$.
\begin{itemize}
\item In the case $k=0$, a class $A \in H_{2}(M;\Z)$ is written as $A=p\hat{B}+qF$ for $p,q \in \Z$. Assume that $A$ satisfies the three conditions. Since $A \cdot A=0$, we get, using Facts \ref{h2facts}, that  if $M=\Sigma \times S^2$ then $2pq=0$, i.e., either $p=0$ or $q=0$, and  if $M=M_{\Sigma}$ then $0=p^2+2pq=p(p+2q)$, i.e.,  either $p=0$ or $p+2q=0$. 
If $M=\Sigma \times S^2$ and $q=0$ then by the third property of $A$ we have  $2=c_{1}(TM)(A)=(2-2g)p$. Similarly, if  $M=M_{\Sigma}$ and $p+2q=0$ then $2=c_{1}(TM)A=(2-2g)p+p+2q=(2-2g)p$. Since $g$ is a positive integer, the equality $2=(2-2g)p$ holds only if $g=2$ and $p=-1$, however, if $M=\Sigma \times S^2$ this (and $q=0$) yield that  $\omega(A)= -\omega(\hat{B})<0$ contradicting the first condition; if $M=M_{\Sigma}$ this (and $p+2q=0$) yield that $2q=1$ contradicting the fact that $q$ is an integer.
 We conclude that $p=0$ hence $2q=c_{1}(TM)(A)=2$, i.e., $q=1$.

\item Assume $k \geq 1$. Let $A \in H_{2}(M;\Z)$. By item (1) in Facts \ref{h2facts}, $A=p\hat{B}+qF-\sum_{i=1}^{k}r_i E_i$ for $p,q \in \Z$. If $A$ satisfies the first condition, then, by \eqref{Biran} and item(2) in Facts \ref{h2facts}, for all $1\leq i \leq k$ we have $r_i=A \cdot E_{i}=0$ hence $A=p\hat{B}+qF$, so $p=A \cdot (F-E_{1})=0$. Since $2=c_{1}(TM)(A)=qc_{1}(TM)(F)=2q$, we get $q=1$. 

On the other hand, by \eqref{Biran} and Facts \ref{h2facts}, the class $F$ satisfies the conditions.
\end{itemize}
\end{proof}

%%%%%%%%%%%%%%%%%%%%%%%%
\subsection*{Necessary and sufficient conditions for a vector to encode a blowup form}
%%%%%%%%%%%%%%%%%%%%%%%%

We now turn to the question of when  a vector 
$(\lambda_F,\lambda_B;\delta_{1},\ldots,\delta_{k})\in \R^{2+k}$
encodes the cohomology class of a blowup form.

\begin{Lemma} \labell{poslemma}
Assume that $\Sigma$ is of positive genus. Let $k \geq 0$.
A vector $(\lambda_F,\lambda_B ; \delta_1 , \ldots , \delta_k)$
encodes the cohomology class of a blowup form $\omega$ on $M=(\Sigma \times S^2)_{k}$ ($M=(M_{\Sigma})_{k})$ if and only if
\begin{itemize}
\item the numbers $\lambda_F,\lambda_B;  \delta_1 , \ldots , \delta_k$ are positive;
\item $\lambda_F > \delta_i $ for all $i$; and 
\item the volume inequality
$\lambda_F \lambda_B - \half(\delta_1^2 + \ldots + \delta_k^2) > 0$ holds.
\end{itemize}
\end{Lemma}

\begin{proof}
If the vector $(\lambda_F,\lambda_B ; \delta_1 , \ldots , \delta_k)$ encodes the
cohomology class of a blowup form $\omega$, then
the symplectic areas of the embedded $\omega$-symplectic exceptional spheres
representing the classes in $\calE(M)=\{E_1,\ldots,E_k, F-E_1,\ldots,F-E_k\}$
and the symplectic volume of $M$ are necessarily positive.  This establishes the
fact that $\delta_i>0$, $\lambda_F>\delta_i$ and the volume inequality.
The definition of a blowup form requires $\lambda_F>0$, including for the case $k=0$. 
Finally becuase the symplectic volume and $\lambda_F$ are positive, 
$\lambda_B>0$ as well.

The fact that the listed conditions on the vector 
$(\lambda_F,\lambda_B ; \delta_1 , \ldots , \delta_k)$
are sufficient to guarantee the existence of a blowup form follows from 
Li and Liu's characterization of symplectic forms with a standard canonical class on blowups of ruled 
symplectic manifolds \cite[Theorem~3]{Li-Liu}.  They prove that for any
symplectic $4$-manifold with $b^+=1$ (which includes all blowups of ruled surfaces), 
the symplectic cone of symplectic forms for which the first Chern class is the same as that of a blowup form (i.e., encoded as in item (3) of Facts \ref{h2facts}) is described by
$$
\mathcal{C}_{0} = \left\{ \alpha\in H^2(M;\R)\ \Big|\  
\begin{array}{c} \alpha\smile\alpha>0 \mbox{ \bf and }\\
\alpha(E)>0\ \forall \ 
 \text{ ``exceptional class" } E \text{ s.t. } c_{1}(TM)(E)=1\end{array} \right\}.
$$
 By ``exceptional class"
Li and Liu refer to a homology class $E$ that is represented by a smoothly embedded
sphere with self intersection $-1$.  The facts that 
\begin{itemize}
\item a class $E$ is ``exceptional" with  $c_{1}(TM)(E)=1$ if and only if $E \in \calE(M)$; 
and 
\item every symplectic form
with first Chern class as that of a blowup form is a blowup form
\end{itemize}
follow from results that are given in 
\cite[Lemma 3.5 Part 2]{Li-Liu}, 
and \cite[Theorem A]{liliu-ruled}; see also the explanation in \cite[Section 6]{blowups}.
Thus, any cohomology class encoded by a vector 
$(\lambda_F,\lambda_B ; \delta_1 , \ldots , \delta_k)$
satisfying the conditions of the lemma is  in $\mathcal{C}_{0}$ and so is 
the cohomology class of some blowup form $\omega$.  We will  call
$\mathcal{C}_{0}$ the \textbf{standard symplectic cone}.
\end{proof}

\begin{Remark}\label{rem:G width}
Note that  Lemma~\ref{poslemma} may be used to determine the {\bf Gromov width} of a symplectic 
$4$-manifold $M=(\Sigma \times S^2)_{k}$ or $M=(M_{\Sigma})_{k}$ with $g(\Sigma)>0$ and $k\geq 0$.
To do so, we must determine for which positive numbers $\alpha$ there is a 
symplectic embedding of a ball of capacity $\alpha$ into $M$.  
This is equivalent to determining if there is a symplectic blowup of $M$ of size $\alpha$.
We start with a vector $(\lambda_F,\lambda_B ; \delta_1, \ldots , \delta_k)\in\R^{k+2}_{>0}$ 
that encodes a blowup form on $M$.
Lemma~\ref{poslemma} tells us that there exists a blowup of $M$ of size $\alpha$ 
so long as $0<\alpha<\lambda_F$ and the volume inequality
$$
 \lambda_F \lambda_B - \half \left( \alpha^2 + \sum_{i=1}^{k}{\delta_i}^2\right)>0.
$$
is satisfied. The Gromov width of $M$ is the $\mathrm{sup}$ of all such $\alpha$.
Hence, the Gromov width will be determined by $\lambda_F$ and the volume of $M$.
\end{Remark}

\begin{Remark}\label{rem:packing}
Lemma~\ref{poslemma} may also be used to determine the {\bf packing number}
of a symplectic 
$4$-manifold $M=(\Sigma \times S^2)_{k}$ or $M=(M_{\Sigma})_{k}$ with $g(\Sigma)>0$ and $k\geq 0$.
Following \cite[Definition~2.C]{biran:packing}, the packing number of $(M,\omega)$ is
$$
P_{(M,\omega)} = 1+\max\{ N\in\N \ |\ \mbox{there does not exist a full packing of $M$ by $N$ equal balls}\}.
$$
For such $M$ with a blowup form encoded by the vector 
$(\lambda_F,\lambda_B ; \delta_1, \ldots , \delta_k)$, 
it is straight forward to show that
$$
P_{(M,\omega)} = \left\lceil \frac{2\cdot \left( \lambda_F \lambda_B - \half  \sum_{i=1}^{k}{\delta_i}^2
\right)}{\lambda_F^2} \right\rceil.
$$
In the special case when $k=0$ and $M=S^2\times\Sigma$, this is Biran's result
\cite[Corollary~5.C]{biran:packing}.
\end{Remark}

\subsection*{Symplectic blow down}
By Weinstein's tubular neighborhood theorem \cite{W}, a neighborhood of an exceptional sphere is symplectomorphic
to a neighborhood of the exceptional divisor in a standard blowup of $\C^2$.
We can then {\bf blow down}  along $C$ and get a symplectic manifold 
whose symplectic blowup is naturally isomorphic to $(M,\omega)$.

\begin{Lemma}[Uniqueness of blow-downs] \labell{bdownunique}
Let $(M,\omega)$ be a closed  symplectic four-manifold. 
Let $C_1$ and $C_2$ be embedded $\omega$-symplectic spheres of self
intersection $-1$. Assume that $C_1$ and $C_2$ are in the same homology
class. Let $(N_1,\omega_1)$ and $(N_2,\omega_2)$ be  symplectic
four-manifolds that are obtained by blowing down $(M,\omega)$ along
$C_1$ and $C_2$, respectively. Then there is a symplectomorphism between
$(N_1,\omega_1)$ and $(N_2,\omega_2)$ that induces the identity map on
the second homology 
with respect to the decompositions\newline $H_{2}(M)=H_{2}(N_i) \oplus \Z [C_i]$.
\end{Lemma}
\noindent For a proof of the uniqueness of blow downs in four dimensions, see \cite[Lemma A.1]{ke}.

As a result of Lemma~\ref{lemequiv}, 
Lemma~\ref{bdownunique}, the definition of a blowup form, and Remark \ref{rem:blowup}, we get the following lemma. 
\begin{Lemma} \labell{blowdown Ek}
Let  $M=(\Sigma \times S^2)_{k}$ (or $M=(M_{\Sigma})_{k}$), 
let $\omega$ be a blowup form on $M$, and
let the vector $(\lambda_F,\lambda_B; \delta_1 , \ldots , \delta_{k-1}, \delta_k)$
be the one encoding the cohomology class $[\omega]$.
Then 
\begin{enumerate}
\item
there exists a blowup form $\ol{\omega}$ on $(\Sigma \times S^2)_{k-1}$ (resp.\ $(M_{\Sigma})_{k-1}$)
whose cohomology class is encoded by the vector
$(\lambda_F,\lambda_B;\delta_1, \ldots, \delta_{k-1})$;
\item
for every embedded $\omega$-symplectic sphere in the class $E_k$,
blowing down along it yields a manifold
that is symplectomorphic to $((\Sigma \times S^2)_{k-1},\ol{\omega})$ (resp.\ $((M_{\Sigma})_{k-1},\ol{\omega})$).
\end{enumerate}
\end{Lemma}

%%%%%%%%%%%%%%%%%%%%%%%%%%%%%%%%%%%%%%%%%%%
\section{Hamiltonian circle actions on ruled symplectic\\ four-manifolds with one or two blowups}\label{se: at most 2 blowups}
%%%%%%%%%%%%%%%%%%%%%%%%%%%%%%%%%%%%%%%%%%%

In this section we use the correspondence between Hamiltonian circle actions and decorated graphs to deduce results on symplectic manifolds from combinatorial observations.
We first consider Hamiltonian $S^1$-actions on a symplectic manifold obtained from a ruled symplectic manifold by a single blowup. 
\begin{Lemma}\labell{dual}
Assume that $0<\eps<\min\{r,s\}$.
\begin{enumerate}
\item There is a Hamiltonian $S^1$-action on $(\Sigma \times S^2)_{1}$ with the blowup form encoded by $(s,r;\varepsilon)$, for which there exists an equivariant sphere $C$ in the class $E_1$ and an equivariant sphere $D$ in the class  $F-E_{1}$, and blowing down along $C$ yields a Hamiltonian $S^1$-action on $\Sigma \times S^2$ with the compatible symplectic form encoded by $(s,r)$, and blowing down along $D$ yields a Hamiltonian $S^1$-action on $M_{\Sigma}$ with the compatible symplectic form encoded by $(s,r+\half s-\varepsilon)$.

\item There is a Hamiltonian $S^1$-action on $(M_{\Sigma})_{1}$ with the blowup form encoded by $(s,r;\varepsilon)$, for which there exists an equivariant sphere $C$ in the class $E_1$ and an equivariant sphere $D$ in the class  $F-E_{1}$, and blowing down along $C$ yields a Hamiltonian $S^1$
action on $M_{\Sigma}$ with the compatible symplectic form encoded by $(s,r)$, and blowing down along $D$ yields a Hamiltonian $S^1$-action on $\Sigma \times S^2$ with the compatible symplectic form encoded by $(s,r+\half s-\varepsilon)$. 
\end{enumerate}
\end{Lemma}

\begin{proof}
To each case of the lemma, it suffices to inspect the decorated graphs given in Figure \ref{fig:2circ blow ups}. 

\centering{\begin{figure}[h]
\includegraphics[width=\textwidth]{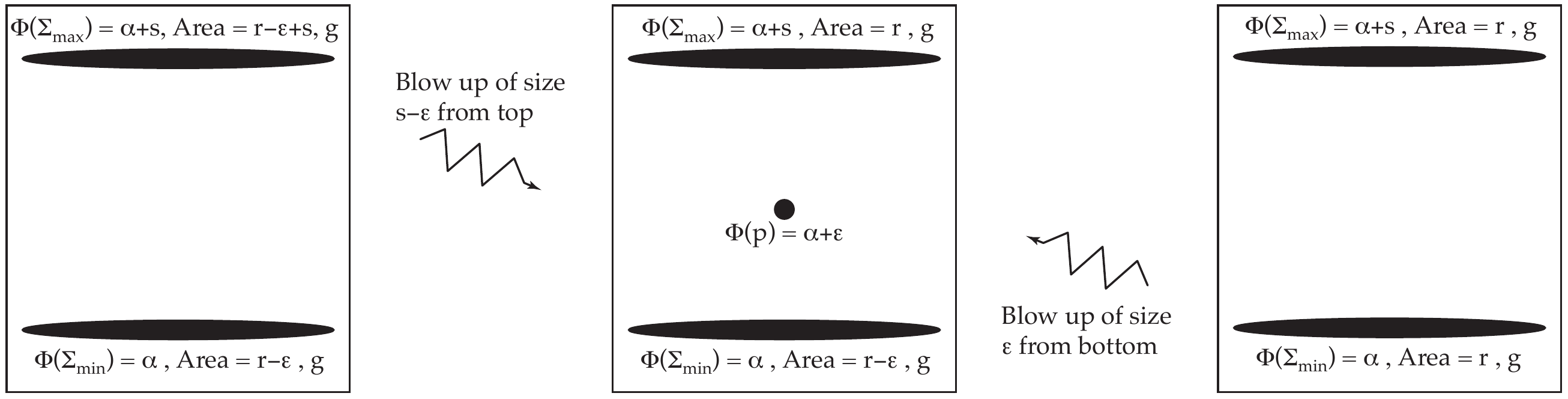} 
\caption{A circle blow up of size $s-\varepsilon$ at the maximal surface of 
$(M_{\Sigma},\omega_{s,r+\frac{s}{2}-\varepsilon})$, on the left, is isomorphic to a
circle blow up of size $\varepsilon$ at the minimal surface of $(\Sigma \times S^2, \omega_{s,r})$, on the right.}
\label{fig:2circ blow ups}
 \end{figure}}

\end{proof}

\noindent Together, Lemma \ref{dual} and Lemma \ref{bdownunique} allow us to prove the following result.
\begin{Corollary} \labell{corbdowne12 prime}
Let $M=(\Sigma \times S^2)_{1}$ $(\mbox{resp. } 
M=(M_{\Sigma})_{1})$ with a blowup form $\omega$ encoded 
by $(\lambda_F,\lambda_B;\delta)$. 
For every embedded $\omega$-symplectic sphere in the class 
$F-E_{1}$, blowing down along it
yields a symplectic manifold
that is symplectomorphic to $M_{\Sigma}$ $(\mbox{resp. }\Sigma \times S^2)$ with the symplectic ruling encoded by $({\lambda_F,\lambda_B+\half \lambda_F-\delta})$.  
\end{Corollary}

\begin{proof}
By Lemma \ref{dual}, for some blowup form on $M$ there is an embedded symplectic sphere in the exceptional class $F-E_{1}$ such that blowing down along it yields the stated symplectic manifold. By Lemma \ref{calE and J}, for every blowup form on $M$ there exists a symplectic sphere in $F-E_{1}$, and by Lemma~\ref{lemequiv}, blowing down along it has the same effect on homology. The blow
down is well-defined by Lemma~\ref{bdownunique}, and hence yields the desired symplectic manifold.
\end{proof}

\begin{Corollary}\labell{cor-same}
Let $k \geq 1$. The symplectic manifold $(M_{\Sigma})_{k}$ $(\mbox{resp. } 
(\Sigma \times S^2)_{k})$ with a blowup form $\omega$ with $[\omega]$ encoded by $(\lambda_F,\lambda_B;\delta_1,\ldots,\delta_k)$ is symplectomorphic to $(\Sigma \times S^2)_{k}$ $(\mbox{resp. } (M_{\Sigma})_{k})$ with a blowup form $\omega'$ with $[\omega']$ encoded by $(\lambda_F,\lambda_B+\half \lambda_F-\delta_1;\lambda_F-\delta_1,\delta_2, \ldots,\delta_k)$. 
Moreover if $k\geq 2$ and $[\omega]$ is in $g$-reduced form, then so is $[\omega']$.
\end{Corollary}

We now continue our study by looking at the effect of two blowups on a decorated graph.
In Figures \ref{fig:circle-blowup} and \ref{fig:circle-blowup2} we describe two different Hamiltonian blowups of different ruled symplectic manifolds that yield the same symplectic manifold with Hamiltonian $S^1$-action. Assume that 
$0<\varepsilon_1, \, \varepsilon_2<\min\{r,s\}$.

\begin{figure}[h] 
\centering{
\includegraphics[width=\textwidth]{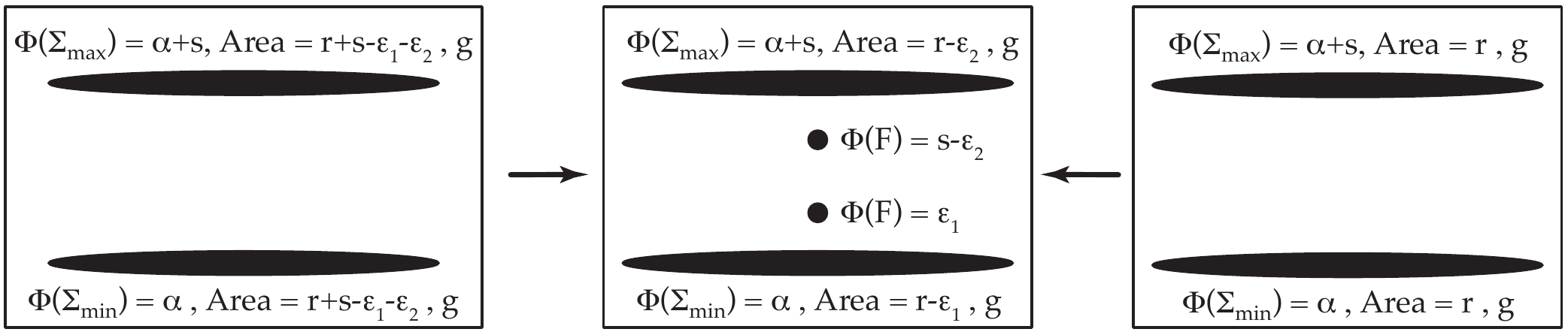}
\caption{On the left, we have a decorated graph for $\Sigma\times S^2$ with $\lambda_B=r+s-\varepsilon_1-\varepsilon_2$ and $\lambda_F=s$.  From the left to the center, we take two blow-ups of $\Sigma\times S^2$ of sizes $s-\varepsilon_2$ from
the bottom and $s-\varepsilon_1$ from the top.\newline 
On the right, we have a decorated graph for $\Sigma\times S^2$ with $\lambda_B=r$ and $\lambda_F=s$.  From the right to the center, we take two blow-ups of $\Sigma\times S^2$ of sizes $\varepsilon_1$ from
the bottom and $\varepsilon_2$ from the top.}
\label{fig:circle-blowup}
}
\end{figure}

\begin{figure}[h] 
\centering{
\includegraphics[width=\textwidth]{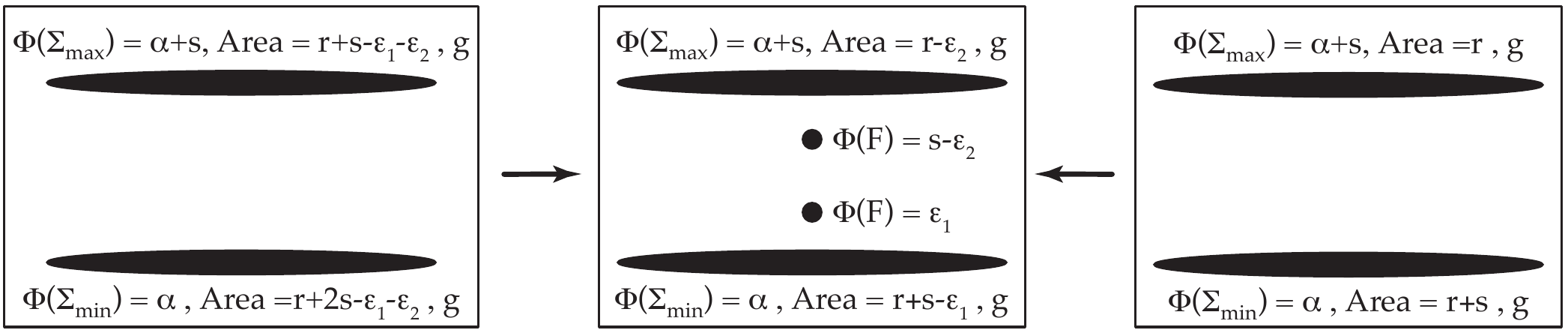}
\caption{{\bf On the left,} we have a decorated graph for $M_\Sigma$ with\newline 
$\lambda_B=r+\frac{s}{2}+s-\varepsilon_1-\varepsilon_2$ and $\lambda_F=s$.  
From the left to the center, we take two blow-ups of $M_\Sigma$ of sizes $s-\varepsilon_2$ from
the bottom and $s-\varepsilon_1$ from the top.\newline 
{\bf On the right,} we have a decorated graph for $M_\Sigma$ with $\lambda_B=r+\frac{s}{2}$ and $\lambda_F=s$.  From the right to the center, we take two blow-ups of $M_{\Sigma}$ of sizes $\varepsilon_1$ from
the bottom and $\varepsilon_2$ from the top.}
\label{fig:circle-blowup2}
}
\end{figure}

\noindent Consequently, we get an analogue of the {\bf Cremona transformation}, adjusted for the positive genus case.
\begin{Definition} \labell{Cremdef}
Let $k \geq 2$.
For a vector $v = (\lambda_F,\lambda_B ; \delta_1, \ldots , \delta_k)$,
let $$\mathbf{defect}(v) := \delta_1 + \delta_2  - \lambda_F.$$
We define
$$\mathbf{cremona}(v) := (\lambda_F',\lambda_B' ; \delta_1', \ldots , \delta_k')$$
by setting
$$\begin{array}{lcl}
\lambda_F' & := & \lambda_F\\
\lambda_B' & := & \lambda_B - \mathbf{defect}(v) \\
\delta_j' & := &
\begin{cases}
 \delta_1- \mathbf{defect}(v)=\lambda_F-\delta_2 & \text{ if } j=1 \\
 \delta_2-\mathbf{defect}(v)=\lambda_F-\delta_1 & \text{ if } j=2 \\
  \delta_j              & \text{ if } 3 \leq j \leq k .
\end{cases}
\end{array} $$
\end{Definition}

\begin{Corollary} \labell{cor:cremona}
Let $k \geq 2$, and consider a $k$-fold blowup $M=(\Sigma \times S^2)_{k}$ or $M=(M_{\Sigma})_{k}$
of a ruled symplectic manifold, and
$v=(\lambda_F,\lambda_B;\delta_1,\ldots,\delta_k)$. 
Let $(\lambda_F',\lambda_B';\delta_1',\ldots,\delta_k')=\mathbf{cremona}(v)$.
Suppose that there exists a blowup form
$\omega_{\lambda_F,\lambda_B;\delta_1,\ldots,\delta_k}$
 whose cohomology class
is encoded by $v$.
Then there exists a blowup form
$\omega_{\lambda_F',\lambda_B';\delta_1',\ldots,\delta_k'}$
whose cohomology class
is encoded by $\mathbf{cremona}(v)$.
Moreover, $(M,\omega_{\lambda_F,\lambda_B;\delta_1,\ldots,\delta_k})$
and $(M,\omega_{\lambda_F',\lambda_B';\delta_1',\ldots,\delta_k'})$
are symplectomorphic.
\end{Corollary}

\begin{proof}
First assume $k=2$.
In  Figures \ref{fig:circle-blowup} and \ref{fig:circle-blowup2} we show that
there exists a blowup form $\omega$ on the manifold $M^{2}:=(\Sigma \times S^2)_{2}$ or on $M^{2}:=(M_{\Sigma})_{2}$,
such that there is a diffeomorphism $\psi \colon M^2 \to M^2$ with the following properties. 
The induced map $\psi_*$ on $H_2$ maps 
$$
F  \mapsto F; \,\,
E_1  \mapsto   F-E_2; \,\,
E_2 \mapsto  F-E_1; \,\,
F-E_1  \mapsto  E_2; \,\,
F-E_2  \mapsto  E_1;\,\,
\hat{B} \mapsto \hat{B}+F-E_1-E_2,
$$
where $\hat{B}=B$ if $M^{2}=(\Sigma \times S^2)_{2}$, and $\hat{B}=B_{-1}$ if $M^{2}=(M_{\Sigma})_{2}$.
Moreover $\psi^{*}\omega$ is a blowup form.

In particular, there exists a blowup form on $M^{2}$ such that the pullback by $\psi$ is a blowup form. By Lemma \ref{lemequiv}, it follows that for {\bf every}  blowup form, its pullback by $\psi$ is a blowup form.
Because of how $\psi$ acts on homology,
if $\omega$ is a symplectic form 
that is encoded by the vector $(\lambda_F,\lambda_B; \delta_1, \delta_2)$,
then $\psi^*\omega$ is a symplectic form
that is encoded by the vector $(\lambda'_F,\lambda'_B; \delta_1', \delta_2')$.
The case $k \geq 3$ then follows by the uniqueness of symplectic blowups of given sizes.
\end{proof}

%%%%%%%%%%%%%%%%%%%%%%%%%%%%%%%%%%%%%%%%%%%%%%%%%%%%%%%%%%%%%%%
\section{Existence of $g$-reduced form} \labell{sec:reduce}
%%%%%%%%%%%%%%%%%%%%%%%%%%%%%%%%%%%%%%%%%%%%%%%%%%%%%%%%%%%%%%%

In this section we prove the existence part of Theorem \ref{thmexu}.  
That is, given   
a $k$-fold blowup of a ruled symplectic manifold over a surface of positive genus,  
$M=(\Sigma \times S^2)_{k}$ or $M=(M_{\Sigma})_{k}$, we show that there
exists a blowup form on $M$ whose cohomology class is $g$-reduced.
We achieve this by providing an algorithm that
starts with a vector encoding the cohomology class of a blowup form 
on a $k$-fold blowup of an irrational ruled manifold, and 
returns a symplectomorphic blowup form which is $g$-reduced.

\begin{Definition}
Let $k \geq 2$. The {\bf Cremona move} on ${\R}^{2+k}$ is the composition of the following two maps:
\begin{enumerate}
\item first the map
$$ v \mapsto \begin{cases} \mathbf{cremona}(v) & \text{ if }\defect(v)> 0\\
                                                   v & \text{ otherwise.} \end{cases}; \mbox{ followed by }$$
\item the map $ \mathbf{sort}: (\lambda_F, \lambda_B;\delta_1,\ldots,\delta_k) \mapsto (\lambda_F,\lambda_B;\delta_{i_{1}},\ldots,\delta_{i_{k}}),$ 
which permutes the last $k$ entries such that $\delta_{i_{1}}\geq \ldots \geq \delta_{i_{k}}$.
\end{enumerate}
\end{Definition}

\noindent Let $k \geq 2$. It is straight forward to check that the Cremona move satisfies the 
following properties.

\begin{Lemma}\labell{noac}
Let $v=  (\lambda_F,\lambda_B ; \delta_1, \ldots , \delta_k)$ 
be a vector in ${\R}^{2+k}$.
The set of real numbers $\delta_i'$ 
that occur 
among the last $k$ entries
in vectors $v' = (\lambda'_F,\lambda'_B;\delta_1',\ldots,\delta_k')$
that can be obtained from $v$
by finitely many iterations of the Cremona move is a nonempty subset of the finite set $\{\delta_1,\ldots,\delta_k, \lambda_F-\delta_1,\ldots,\lambda_F-\delta_k\}$,
and in particular
has no accumulation points.
\end{Lemma}

\begin{Lemma} \labell{cremonaisgood}
\begin{enumerate}
\item The Cremona move preserves the standard symplectic cone
$$
\mathcal{C}_{0}=\left\{(\lambda_F,\lambda_B;\delta_1,\ldots,\delta_k) \in \R_{>0}^{2+k} \,\Bigg|\, \lambda_F-\delta_i>0 \text{ for all } i \text{ and }
\lambda_F \lambda_B - \half \sum_{i=1}^{k}{\delta_i}^2>0 \right\}.$$
\item If the vector $v'=(\lambda_F,\lambda'_B;\delta'_1,\ldots,\delta'_k)$
is obtained from $v=(\lambda_F,\lambda_B;\delta_1,\ldots,\delta_k)$ by the Cremona move and $v'$ is different from $v$, then 
 $\delta'_{i}\leq \delta_{i}$ for all $i$, and for at least one $i$ we have $\delta'_{i}<\delta_{i}$.
\item A vector $v=(\lambda_F,\lambda_B;\delta_1,\ldots,\delta_k) \in \mathcal{C}_{0}$ 
 is fixed by the Cremona move if and only if it satisfies equation \eqref{eq:red5}.
\end{enumerate}
\end{Lemma}

\begin{proof}[Sketch of Proof]
Let $v=(\lambda_F,\lambda_B;\delta_1,\ldots,\delta_k) \in \mathcal{C}_{0}$ and $v'=(\lambda_F,\lambda'_B;\delta'_1,\ldots,\delta'_k)$ the vector obtained from $v$ by the Cremona move.
In proving part (1), we note that both the ``volume" $\lambda_F \lambda_B - \half \sum_{i=1}^{k}{\delta_i}^2$ and $\lambda_F$ are preserved by the Cremona move. 
Hence, since the ``volume" of $v'$ and $\lambda'_F=\lambda_F$ are positive and all other terms in the ``volume" are nonpositive, we get that $\lambda'_B>0$ as well.
Also, by Lemma \ref{noac} and since $v \in \mathcal{C}_{0}$, for all $i$, the numbers $\delta'_i$ and $\lambda_F-\delta'_i$ are positive as well. 
Part (2) and part (3) are immediate from the definition of the Cremona move.
\end{proof}

\begin{Notation}
Let $g>0$. 
A vector in  ${\R}^{2+k}$ ($k \geq2$) with all positive entries encodes a $g$-reduced cohomology class if it satisfies equation \eqref{eq:red5}; we will call such a vector {\bf $g$-reduced}.
\end{Notation}

\noindent Assume $k \geq 2$ and $g>0$. We will now use the Cremona move to obtain a $g$-reduced vector.

\begin{algorithm}[h]
\caption{ --- \textsc{CremonaReduce}.     \newline
$\circ$  Input  a vector $\mathbf{v}=(\lambda_F,\lambda_B ; \delta_1, \ldots , \delta_k) \text{ in the standard symplectic cone in }\R^{k+2}$ \newline 
$\circ$  Output a  {\bf $g$-reduced} vector $(\lambda_F',\lambda_B' ; \delta_1', \ldots , \delta_k') \text{ in the standard symplectic cone in }\R^{k+2}$}\label{alg:cremona}
	\begin{algorithmic}[1]
	         \State $\mathbf{sort}\big(\mathbf{v}\big)$
		\State \textbf{while} $\mathbf{defect}(\mathbf{v})> 0$ \textbf{do}
		\State \qquad \textbf{set} $\mathbf{v}= \mathbf{sort}\Big(   \mathbf{cremona}\big(\mathbf{v}\big) \Big)$
%		\State Add {\bf halt} if $\delta_k\leq 0$ or $\delta_k<0$?
		\State \textbf{end while loop}
		\State
		\State \textbf{return} $\mathbf{v}$
	\end{algorithmic}
\end{algorithm}

\begin{Lemma}\labell{finitesteps2}
The algorithm \textsc{CremonaReduce} terminates after finitely many steps.
\end{Lemma}

\begin{proof}
We prove this by contradiction.
Suppose that this process does not terminate for some vector $v$ in
the standard symplectic cone.
Then we get an infinite sequence $v^{(n)}$, for any $n \in \N$,
of vectors in the standard symplectic cone in ${\R}^{2+k}$.  That is, $v^{(1)}=\mathbf{sort}(v)$
and $v^{(n+1)}=\mathbf{sort}\left(\mathbf{cremona}\left(v^{(n)}\right)\right)$.
By part (2) of Lemma \ref{cremonaisgood}, 
for some $1 \leq i_0 \leq k$, the sequence ${\delta^{(n)}_{i_0}}$
has an infinite subsequence ${\delta^{(n_{\ell})}_{i_0}}$
that is strictly decreasing. 
Moreover the numbers ${\delta^{(n_{\ell})}_{i_0}}$ are all bounded below, as they are positive. 
Thus, the subsequence ${\delta^{(n_{\ell})}_{i_0}}$ has an accumulation point,
contradicting
Lemma~\ref{noac}.
\end{proof}

\begin{Remark}\label{remexproof}
Lemma \ref{finitesteps2} implies that Algorithm~\ref{alg:cremona} terminates. By item (3) of Lemma \ref{cremonaisgood} it 
produces a $g$-reduced vector. Moreover, by Corollary \ref{permute diff top} and Corollary \ref{cor:cremona}, if the input vector encodes the cohomology
 class of a blowup form $\omega$ on $M=(\Sigma \times S^2)_{k}$ or $M=(M_{\Sigma})_{k}$, then the output vector 
 encodes the cohomology class of a blowup form on $M$ that is diffeomorphic to $\omega$. 
This completes the
proof of the existence portion of Theorem~\ref{thmexu}. 
\end{Remark}

\begin{Remark}\labell{notapush}
In the case when  $g(\Sigma)=0$, i.e., $\Sigma=S^2$, if we apply Algorithm~\ref{alg:cremona} to a vector 
$(\lambda_F,\lambda_B ; \delta_1, \ldots , \delta_k)$ encoding a blowup form on 
$M=(S^2 \times S^2)_{k}$ or $M=(M_{S^2})_{k}$, the output vector $(\lambda'_F, \lambda'_B;\delta'_1,\ldots,\delta'_k)$ may 
encode a blowup-up form that is not $o$-reduced.  The trouble is that the condition 
$$
\left\{\begin{array}{ll}
\lambda'_F\leq \lambda'_B &\mbox{ if } M=(S^2 \times S^2)_{k}; \mbox{ or } \\
(\half \lambda'_F+\delta'_1)\leq \lambda'_B & \mbox{ if } M=(M_{S^2})_{k},
\end{array}\right.
$$ 
may not hold. 
In particular, Algorithm~\ref{alg:cremona}  
is {\bf not} the push forward of the algorithm in \cite[\S 2.16]{blowups} to obtain a $0$-reduced form given a 
blowup form on $(M_{S^2})_{k}$ under the map
$$(\lambda;\delta,\delta_1,\dots,\delta_k) \mapsto (\lambda_F,\lambda_B;\delta_1,\ldots,\delta_k)$$ 
 as in \eqref{eqlamdel}. 
\end{Remark}

% ========================================================
\section{Uniqueness of the $g$-reduced form}
\labell{sec:minimal set}
% ========================================================

Our goal in this section is to prove the uniqueness part of Theorem \ref{thmexu}.
 Let $M=(\Sigma \times S^2)_{k}$ or $M=(M_{\Sigma})_{k}$; assume that $g(\Sigma)>0$.
For any blowup form $\omega$ 
the set of exceptional classes of minimal area in $(M,\omega)$
only depends on the cohomology class $[\omega]$.
We denote this set by 
$$ \calE^v_{\min} ,$$
where $v \in \R^{2+k}$ is the vector
that encodes the cohomology class $[\omega]$.
If $k\geq 2$ and $[\omega]$ is in $g$-reduced form, then by  \eqref{Biran}, 
$E_k\in\calE^v_{\min}$.
We proceed to identify \emph{all} the elements of $\calE^v_{\min}$.

\begin{Notation}\labell{notuni}
 Let
$v = (\lambda_F,\lambda_B;\delta_1,\ldots,\delta_k)$ be the vector encoding the cohomology class of a blowup form $\omega$ on $M=(\Sigma \times S^2)_{k}$ or $M=(M_{\Sigma})_{k}$. If $k \geq 2$, assume that $[\omega]$ is in $g$-reduced form. 
In particular, $\delta_1 \geq \ldots \geq \delta_k$, and the
numbers $\lambda_F, \, \lambda_B, \delta_1,\ldots,\delta_k$ 
are positive.

We denote by
$j_v $
the smallest nonnegative integer $j$ such that
$\delta_{j+1} = \ldots = \delta_k$.  Thus, either $j_v=0$
and $\delta_1 = \ldots = \delta_k$,  or $1 \leq j_{v} \leq k-1$
and
$$ \delta_1 \geq \ldots \geq \delta_{j_v} > \delta_{{j_v}+1} 
   = \ldots = \delta_k.$$
\end{Notation}

As an immediate corollary of the identification of the exceptional classes in \eqref{Biran}, 
and the definition of a $g$-reduced form, we have the following proposition.

\begin{Proposition} \labell{thm:Emin} 
Let $(M,\omega)$ and $v$ be as in Notation \ref{notuni}.
Assume that $g(\Sigma)>0$.
\begin{itemize}
\item In the case of one blowup, $k=1$, we have the following possibilities.
\begin{enumerate}
\item Suppose that $\delta_1<\half \lambda_F$. Then $\calE^v_{\min}=\left\{E_{1}\right\}$, so  $\sharp \calE^v_{\min}=k$.
\item Suppose that $\delta_1>\half \lambda_F$. Then $\calE^v_{\min}=\left\{F-E_{1}\right\}$, so  $\sharp \calE^v_{\min}=k$.
\item Suppose that $\delta_1=\half \lambda_F$. Then $\calE^v_{\min}=\left\{E_{1}, F-E_{1}\right\}$,  so $\sharp \calE^v_{\min}=2k$.
\end{enumerate}

\item For $k \geq 2$ blowups, we have the following possibilities.
\begin{enumerate}
\item
Suppose that $\delta_1 \leq \half \lambda_F$.

   \begin{enumerate}
   \item Suppose that $\delta_k <\half \lambda_F$.
         Then $\calE^v_{\min} = \left\{ E_{j_v+1} , \ldots , E_k \right\}$, so  $\sharp \calE^v_{\min} \leq k$.
   \item Suppose that $\delta_k = \half \lambda_F$.
         Then  $\calE^v_{\min} = \bigcup_{i=1}^{k}\left\{E_i\right\} \cup \bigcup_{i=1}^{k}\left\{F-E_i\right\}$, so  $\sharp \calE^v_{\min}=2k$.
   \end{enumerate}
\item 
Suppose that $\delta_1 >\half \lambda_F$.
     \begin{enumerate}
   \item Suppose that $\lambda_F-\delta_1> \delta_k$.
         Then $\calE^v_{\min} = \{ E_{j_v+1}, \ldots, E_k \} $  and $j_v>0$, so  $\sharp \calE^v_{\min}\leq k-1$.
   \item Suppose that $\lambda_F-\delta_1= \delta_k$.
         Then $\calE^v_{\min} = \{ F-E_1, E_2, \ldots, E_k \} $, so  $\sharp \calE^v_{\min}=k$.
   \end{enumerate}
\end{enumerate}

\end{itemize}

\end{Proposition}

\noindent For the uniqueness of a $g$-reduced form, we will also need the following 
observations on symplectomorphic blowups of ruled manifolds.

\begin{Lemma} \labell{lemsymp again}
Consider $M=(\Sigma \times S^2)_{k}$ or $M=(M_{\Sigma})_{k}$.
Assume that $g(\Sigma)>0$.
Let $\omega$ and $\omega'$ be blowup forms on $M$
whose cohomology classes are encoded by the vectors
$ v = (\lambda_F,\lambda_B;\delta_1,\ldots,\delta_k) $
and
$ v' = (\lambda_F',\lambda_B';\delta_1',\ldots,\delta_k'),$
which are both  $g$-reduced.

Let $\varphi \colon (M,\omega) \to (M,\omega')$
be a symplectomorphism,
and let $\varphi_* \colon H_2(M) \to H_2(M)$
be the induced map on the homology.

\begin{enumerate}
\item  The isomorphism $\varphi_*$ sends $\calE(M)$ to itself.

\item
The isomorphism
$\varphi_*$ sends the set $\calE^{v}_{\min}$
to the set  $\calE^{v'}_{\min}$. If $k \geq 2$ then
\begin{equation} \labell{eqphimin again}
 \delta_k = \delta_k'.
\end{equation}
If $k=1$ then 
\begin{equation} \labell{deltak1again}
\min(\delta_1,\lambda_F-\delta_1)=\min(\delta'_1,\lambda'_F-\delta'_1).
\end{equation}
\item $\displaystyle{\lambda_F=\lambda'_F.}$

\item  
$\displaystyle{
  \lambda_B \lambda_F -\half \sum_{i=1}^{k} {\delta_i}^2
  = \lambda'_B \lambda'_F -\half\sum_{i=1}^{k} {\delta'_i}^2 .
}$
\end{enumerate}
\end{Lemma}

\begin{proof}
Item (1) and the first part of (2)
follow immediately from the fact that $\varphi$ carries $\omega$ to $\omega'$.
The second part of (2) follows from \eqref{Biran} and the ``$g$-reduced'' assumption.
Item (3) follows from Lemma \ref{keepsF}.
Item (4) follows from the fact that the symplectomorphism
preserves symplectic volume, because
$$   \lambda_B \lambda_F -\half \sum_{i=1}^{k} {\delta_i}^2
  = \frac{1}{(2\pi)^2} \int_{M} \frac{\omega \wedge \omega}{2!}  $$
and a similar formula with $\omega'$.
\end{proof}

We now have all the ingredients to prove that the $g$-reduced form is unique.

\begin{proof}[Proof of the uniqueness part of Theorem \ref{thmexu}]
Recall that we have a manifold
$M$ that is a $k$-fold blowup, either $(M_{\Sigma})_{k}$ or $(S^2 \times \Sigma)_{k}$, where
$\Sigma$ has positive genus.
Let $\omega$ and $\omega'$ be blowup forms on $M$
whose cohomology classes are encoded by the vectors
$ v = (\lambda_F,\lambda_B;\delta_1,\ldots,\delta_k) $
and
$ v' = (\lambda'_F,\lambda'_B;\delta_1',\ldots,\delta_k'),$
and, if $k \geq 2$, are both in $g$-reduced form.
Suppose that $(M,\omega)$ and $(M,\omega')$ are symplectomorphic. We aim to show
that $v=v'$.

The case enumeration refers to the case enumeration in Proposition \ref{thm:Emin}.

\medskip\noindent\textbf{Suppose that $k \geq 1$ and  that the size of $\mathbf{\calE}^{v}_{\min}$ equals $2k$.
\ }

\noindent 
By item (2)  of Lemma \ref{lemsymp again},
the set $\calE^{v'}_{\min}$ also has $2k$ elements.
If $k=1$, by Proposition~\ref{thm:Emin}, both $v$ and $v'$ exhibit case (3).
If $k \geq 2$, by Proposition \ref{thm:Emin},
 both $v$ and $v'$ exhibit case (1b). 
In either case, we have that $v = (\lambda_F,\lambda_B ; \frac{\lambda_F}{2} , \ldots , \frac{\lambda_F}{2})$
and $v'=(\lambda'_F,\lambda'_B;\frac{\lambda'_F}{2},\ldots,\frac{\lambda'_F}{2})$.
From~\eqref{eqphimin again} we deduce that $\lambda_F=\lambda'_F$, hence by item (4)  of Lemma \ref{lemsymp again},
(and the fact that $\lambda_F>0$) also $\lambda_B=\lambda'_B$,
and thus $v=v'$.

\medskip\noindent\textbf{Suppose that the size of $\mathbf{\calE}^{v}_{\min}$ equals $k$.
\ }

\noindent By item (2) of Lemma \ref{lemsymp again},
the set $\calE^{v'}_{\min}$ also has $k$ elements.
If $k=1$ then Proposition \ref{thm:Emin} says that the vector $v$ 
also exhibits one of the cases (1) or (2),  and the vector $v'$ exhibits one of the cases (1) or (2).
When $k \geq 2$,  Proposition \ref{thm:Emin} implies that the vector $v$ exhibits one of the cases 
(1a) with $j_v=0$ or (2b); and the vector $v'$ also exhibits one of the cases (1a) or (2b).
For the purpose of what follows, if $k=1$, we will use (1a) to refer to (1), and (2b) to refer to (2).

The classes in $\calE^{v}_{\min}$ 
can be represented simultaneously by 
 $k$
pairwise disjoint embedded $\omega$-symplectic 
spheres. Indeed, by Lemma 
\ref{calE and J}  
 there exists an $\omega$-tamed almost complex structure $J$ such that each of the exceptional classes in $\calE^{v}_{\min}$ 
 is represented by an embedded $J$-holomorphic sphere. Since in cases (1a) and (2b), for each two classes in this set the intersection number is zero, we conclude, by positivity of intersections of $J$-holomorphic spheres \cite[Proposition 2.4.4]{nsmall}, that these $J$-holomorphic spheres are pairwise disjoint. 
 By  item (2) of Lemma \ref{lemsymp again}, the symplectomorphism sends  $k$ pairwise disjoint embedded $\omega$-symplectic spheres in the classes of $\calE^{v}_{\min}$ to  $k$ pairwise disjoint embedded $\omega'$-symplectic spheres in $\calE^{v'}_{\min}$. 
 Moreover, 
 the symplectomorphism between $(M,\omega)$ and $(M,\omega')$ 
 descends to a symplectomorphism between the blowdowns along $k$ such 
 spheres in $\calE^{v}_{\min}$ and along the corresponding $k$ spheres in $
 \calE^{v'}_{\min}$.

If the vector is in case (1a),
blowing down along the $k$ spheres gives a manifold diffeomorphic to
$S^2 \times \Sigma$ if $M=(S^2 \times \Sigma)_{k}$ and to $M_{\Sigma}$ if $M=(M_{\Sigma})_{k}$, by 
Lemma \ref{blowdown Ek}.
If the vector is in case (2b),
blowing down along the $k$ spheres gives a manifold diffeomorphic to $M_{\Sigma}$ if $M=(\Sigma \times S^2)_{k}$ and to $\Sigma \times S^2$ if $M=(M_{\Sigma})_{k}$, by Lemma \ref{blowdown Ek} and Corollary \ref{corbdowne12 prime}.
Because
the manifolds $S^2 \times \Sigma$ and $M_{\Sigma}$ are not diffeomorphic,
either both  $v$ and $v'$ are in case (1a)
or both are in case (2b). In either case,
$ \calE^{v}_{\min} = \calE^{v'}_{\min}.$

By item (3) in Lemma \ref{lemsymp again}, $\lambda_F=\lambda'_F$.
Moreover, if $k\geq 1$ and both vectors are in case (1a) of Theorem~\ref{thm:Emin}, then 
$v = (\lambda_F,\lambda_B;\delta,\ldots,\delta)$
and $v' = (\lambda_F,\lambda'_B;\delta',\ldots,\delta')$,
and if $k \geq 1$ and both vectors are in case (2b),
$v = (\lambda_F,\lambda_B;\lambda_F-\delta,\ldots,\delta)$
and $v' = (\lambda_F,\lambda'_B;\lambda_F-\delta',\ldots,\delta')$.
In each of these cases, we  conclude, by \eqref{eqphimin again}, that $\delta=\delta'.$
Then, for $k \geq0$, by Item (4) of Lemma \ref{lemsymp again} (and the fact that $\lambda_F>0$) we get that $\lambda_B=\lambda'_B$ hence $v=v'$.

\medskip\noindent\textbf{Suppose that $\mathbf{k \geq 2}$
and that the size of $\mathbf{\calE}^{v}_{\min}$ is smaller or equal to $k-1$.
\ }

\noindent By item (2) of Lemma \ref{lemsymp again},
the set $\calE^{v'}_{\min}$ has size at most $k-1$.
Then by Proposition \ref{thm:Emin}
each of the  vectors $v$ and $v'$  either exhibits case (1a) or  case (2a). So
 \begin{equation*} \labell{eminblowdown}
 \calE^{v}_{\min} = \left\{ E_{{j_v}+1} , \ldots , E_k \right\}
  =  \calE^{v'}_{\min} \text{  where }1 \leq  j_v = k-\sharp  \calE^{v}_{\min}.
 \end{equation*}
By \eqref{eqphimin again}, 
we get $\delta_i={\delta'}_i$ for $j_{v} + 1 \leq i \leq k$.

Since $\omega$ and $\omega'$ are blowup forms, the pairwise disjoint
exceptional classes $ E_{{j_v}+1} , \ldots , E_k $
can be represented simultaneously
by pairwise disjoint embedded $\omega$--symplectic  or $\omega'$--symplectic
spheres.
By Lemma \ref{blowdown Ek}, blowing down along these spheres
in $(M,\omega)$ and in $(M,\omega')$
gives manifolds with a blowup form that is encoded by the vector
$\hat{v} = (\lambda_F,\lambda_B;\delta_1,\ldots,\delta_{j})$
and the vector $\hat{v}' = (\lambda'_F,\lambda'_B;\delta'_1,\ldots,\delta'_{j}).$
By ``uniqueness of blowdown" (Lemma \ref{bdownunique}) 
the resulting manifolds are symplectomorphic.
Because the cohomology classes encoded by vectors $\hat{v}$ and $\hat{v}'$ are in $g$-reduced form,
we may proceed by induction.  We note that the case $k=0$ is included in the case where the size
of $\mathbf{\calE}^{v}_{\min}$ equals $k$.
\end{proof}

Since Algorithm~\ref{alg:cremona} and Lemma~\ref{finitesteps2} 
prove that a $g$-reduced form always exists, see Remark \ref{remexproof}, the proof of Theorem~\ref{thmexu},
the existence and uniqueness of a $g$-reduced form, is now complete.

\appendix
\section{Proof of Lemma \ref{martin}}
\label{appendix1}

In the proof we will use the following corollary of the positivity of intersections of $J$-holomorphic curves, see \cite[Proposition 2.3]{algorithm}.

\begin{Lemma} \labell{positivity of intersections corollary}
Let $(M,\omega)$ be a symplectic four-manifold. 
Let $A$ and $B$ be homology classes in $H_2(M)$ 
that are linearly independent (over $\R$).
Suppose that $\GW(B) \neq 0$,
that $c_1(TM)(A) \geq 1$,
and that there exists an almost complex structure $J_0 \in   \calJ_\tau(M,\omega)$
such that the class $A$ is represented by a $J_0$-holomorphic sphere.
Then the intersection number $A \cdot B$ is nonnegative.
\end{Lemma}

\noindent We now give a proof of Lemma \ref{martin}.

\begin{proof}[Proof of Lemma \ref{martin}]
Assume that $g=g(\Sigma)>0$. 
If  $M=(\Sigma \times S^2)_{1}$ then (by \eqref{Biran}) there are two cases:
\begin{itemize}
\item[a.] The class $E_1$ is of minimal area in $\calE(M)$. In this case, denote
$$\hat{B}:=B,\,\,\,\hat{E}:=E_{1}.$$
\item[b.] Otherwise, the class $F-E_{1}$ is of minimal area in $\calE(M)$. 
Denote
$$\hat{B}:=B+F-E_{1}, \,\,\, \hat{E}:=F-E_{1}.$$
\end{itemize}
If $M=(M_{\Sigma})_{1}$, then there are also two cases:
\begin{itemize}
\item[a.] The class $F-E_{1}$ is of minimal area in $\calE(M)$. In this case, denote
$$\hat{B}:=B_{-1}+F-E_{1},\,\,\, \hat{E}:=F-E_{1}.$$
\item[b.] Otherwise,  the class $E_{1}$ is of minimal area in $\calE(M)$. 
Denote
$$\hat{B}:=B_{1}, \,\,\, \hat{E}:=E_{1}.$$
\end{itemize}
Note that the classes $\hat{B}, F, \hat{E}$, form a $\Z$-basis to $H_2(M)$.

Let $\omega$ be a blowup form on $M$. 
For simplicity, we normalize $\omega$ such that $\frac{\omega(F)}{2\pi}=1$,
and we denote $u=\frac{\omega(\hat{B})}{2\pi}$ and $c = \frac{\omega(\hat{E})}{2\pi}$.  Then, since both $E_1$ and $F-E_{1}$ 
are exceptional classes and by the definition of $\hat{E}$ and $\hat{B}$, 
$$ 0< c \leq \half,$$
and $$u >0.$$ Moreover, in case (b), $u > \half$. 
Also, in case (a), $\hat{B} \cdot \hat{B}=0$ and $c_{1}(TM)(\hat{B})=2-2g$ and in case (b), $\hat{B} \cdot \hat{B}=1$ and $c_{1}(TM)(\hat{B})=3-2g$. In both cases, $\hat{B} \cdot F=1$ and $\hat{B} \cdot \hat{E}=0$.
  
Fix $J \in \calJ_\tau(M,\omega)$. 
To start, suppose that $A$ is a homology class that is represented
by a simple $J$-holomorphic sphere. 
Write $$A = p\hat{B} + qF - r\hat{E},$$ with $p,\, q,\, r$ integers.
We claim
 that the coefficient $p$ of $\hat{B}$ is nonnegative
 and that if this coefficient is zero then $A$
 is one of the classes $F$, $\hat{E}$, and $F-\hat{E}$.

Indeed, the adjunction inequality \cite[Cor.~E.1.7]{nsmall} says that
$$ 0 \leq 2 + A \cdot A - c_1(TM)(A). $$
In case (a), $A \cdot A=2pq-r^2$ and $c_{1}(TM)(A)=(2-2g)p+2q-r$, so we get
\begin{equation} \labell{eqmara}
0 \leq 2(p-1)(q-1)-r(r-1)+2gp 
\end{equation}
In case (b), $A \cdot A = p^2+2pq-r^2$ and $c_{1}(TM)(A)=(3-2g)p+2q-r$, so we get
\begin{equation}\labell{eqmarb}
0 \leq 2(p-1)(q-1)-r(r-1)+p^2+2gp-p
\end{equation}

 Because $A$ is represented
by a $J$-holomorphic sphere for an $\omega$-tamed $J$, the symplectic area of $A$ is positive, i.e.,
\begin{equation} \labell{eqmar2}
 0 < up + q - cr.
\end{equation}
Also note that, because $r$ is an integer,
$r (r-1) \geq 0$ with equality only if $r=0$ or $r=1$.

By Lemma \cite[Lemma 1.2]{pinso},
the minimal exceptional class $\hat{E}$ is represented by a simple $J$-holomorphic sphere, hence, and by the assumption on $A$ and the positivity of intersections of $J$-holomorphic curves in almost complex four-manifolds \cite[Proposition 2.4.4]{nsmall}, we get that either $A=\hat{E}$, i.e. $p=0, \,q=1, \, r=1$, or $r=A \cdot \hat{E} \geq 0$. Assume the latter case.

Assume that $p \leq 0$. 
\begin{itemize}
\item In case (a), first assume $2q-r \geq 1$. Thus (since $g$ is positive and $p\leq 0$) we have $c_{1}(TM)(A) = (2-2g)p+2q-r \geq 2q-r \geq 1$, so, by Lemma \ref{positivity of intersections corollary}, we get $p=A \cdot F \geq 0$; hence $p=0$ and, by \eqref{eqmara}, $-2q+2 \geq r(r-1) \geq 0$, which yields, since $q\geq \half (r+1)>0$ and $r(r-1) \geq 0$, that either $q=1$ and $r=0$ or $q=1$ and $r=1$, i.e., either $A=F$ or $A=F-\hat{E}$. Now assume $2q-r<1$, i.e., $2q-1\leq r-1$.
By \eqref{eqmar2}, and since $-p,\, r,\, c\,, u$ are all nonnegative, we have $q>cr-up\geq 0$. So $0 \leq q-1$ and $0 \leq 2(q-1)<(r-1)$. We deduce then that 
\begin{align*}
2(p-1)(q-1)-r(r-1)+2gp \nonumber \\
&< \, 2(p-1)(q-1)-r(2(q-1))+2gp \nonumber \\
&= \, 2(q-1)(p-1-r)+2gp \leq 0 \nonumber
\end{align*} contradicting \eqref{eqmara}.

\item In case (b), first assume $p+2q-r\geq 1$. Thus $c_{1}(TM)(A)=(2-2g)p+p+2q-r \geq p+2q-r \geq 1$, so, by Lemma \ref{positivity of intersections corollary}, we get $p=A \cdot F\geq 0$; hence it follows from \eqref{eqmarb} as before that either $A=F$ or $A=F-\hat{E}$.
Now assume $p+2q-r<1$, i.e., $p+2q-1<r$.
By \eqref{eqmar2}, and since $-p,\, r,\, c$ are all nonnegative and $u>\half$, we have $0<up+q-cr<\half p+q$. So $0 \leq   p+2q-2< r-1$. We deduce then that 
\begin{align*}
2(p-1)(q-1)-r(r-1)+p^2+2gp-p 
&= \,(p-1)(2q-2+p)-r(r-1)+2gp \nonumber \\
&< \, (p-1)(2q-2+p)-r(2q-2+p)+2gp \nonumber \\
&= \, (p-1-r)(2q-2+p)+2gp \leq 0   \nonumber
\end{align*}
contradicting \eqref{eqmarb}.

\end{itemize}

Having proved the claim, we proceed to prove that there exists
an embedded $J$-holomorphic sphere 
in the class $F-\hat{E}$. This will complete the proof.

By Lemma \ref{calE and J}, we have $\GW(F-\hat{E}) \neq 0$, hence by Gromov's compactness theorem, we can write
$F-\hat{E}= C_1 + \ldots + C_N$ where each $C_i$ is represented
by a simple $J$ holomorphic sphere.
We would like to show that $N=1$;
the simple $J$ holomorphic sphere in the class $F-\hat{E}=C_1$
would then be embedded by the adjunction formula  \cite[Cor.~E.1.7]{nsmall}.
Write each $C_i$ as a combination of $\hat{B}$, $F$, and $\hat{E}$. 
By the previous claim, all the coefficients of $\hat{B}$ are nonnegative.
Because they sum to zero, they must all be equal to zero.
Applying the previous claim again, every $C_i$
is either $F$, or $\hat{E}$, or $F-\hat{E}$.
In particular, all the coefficients of $F$ are nonnegative integers.
Because they sum to $1$, the coefficient of $F$ in all but one of the $C_i$-s are equal to zero, and in one of the $C_i$-s, say in $C_1$, the coefficient of $F$ is equal to $1$, so $C_1$ is either $F$ or $F-\hat{E}$. Because the coefficients of $\hat{E}$ sum to $-1$, at least one of the $C_i$-s is $F-\hat{E}$. We conclude that $C_1=F-\hat{E}$, and $N=1$.   
\end{proof}

\section{Algorithm to Count Circle Actions}\label{tair}

\begin{center}
{\bf by Tair Pnini}\footnote{\textsc{Dept.\ of Mathematics, Physics, and Computer Science, University of Haifa,
at Oranim, Tivon 36006, Israel}\newline \textit{E-mail address}: \texttt{tair.pnini@gmail.com}}
\end{center}

\subsection*{Introduction}
This is a ``greedy" algorithm.
It starts by creating the set of Graphs determined by $\lambda_F, \lambda_B$.
Then, at each stage \textsc{BlowupGraph} creates the set of Graphs (up to equivalence) which may be obtained by performing a 
blowup in the given size on a Graph from the set received in the previous step. 
The output of the program is the number of Graphs in the set  at the final stage.
In order to optimize the program, we use data structures designed to reduce the number of tests for equivalence between pairs of
Graphs and to optimize the equivalence test.

\subsection*{Data Structure}
For each Graph we arbitrarily define one of the two fat vertices to be the bottom fat vertex and the other to be the top fat vertex. In addition, we keep the distance between these two vertices and define it as the \textbf{Height} of the Graph, $\lambda_F$ in this case.
We relate to different paths between the two fat vertices as \textbf{Chains}. A Chain will be represented by a doubly linked list of rational positive numbers, consists of the serial representation of the path, built from vertices and edges that appear one by one alternately. Each vertex is encoded by its distance from the bottom fat vertex and each edge is encoded by its label. The list representing each Chain starts with the lowest non-fat vertex in the Chain and ends with the highest non-fat vertex in the Chain (since edges between a fat vertex and a non-fat vertex are all labeled 1, we do not need to store them). For each list we will keep its first and its last nodes. In addition, every Chain in the Graph will receive a unique serial ID number, so that in a Graph with $n$ Chains the IDs would be from $0$ to $(n-1)$. The Chains of a Graph are kept in an array so that each Chain is put in its ID number cell of the array.

We define two different sorting methods on Chains; \textsc{SortByStart} and \textsc{SortByEnd}. Sorting by start is a lexicogaphic sort in which each node of the list is considered a letter. Sorting by end is defined by the same principle, except that the lists is read backwards and the value encoding each vertex is its distance from the top fat vertex (equals the Height of the Graph minus its distance from the bottom). Notice that sorting the Chains of a Graph by end will give the same result as sorting the Chains of the flipped Graph by start.
For each Graph we will keep two additional arrays, keeping the IDs of the Chains. Both arrays keep the IDs in order, according to the corresponding Chains sort; one according to the ``sort by start" and the other according to the ``sort by end". See Figure \ref{fig:data structure}.

\begin{figure}[ht]
\begin{tabular}{@{}c@{}}{\includegraphics[height=7.1cm]{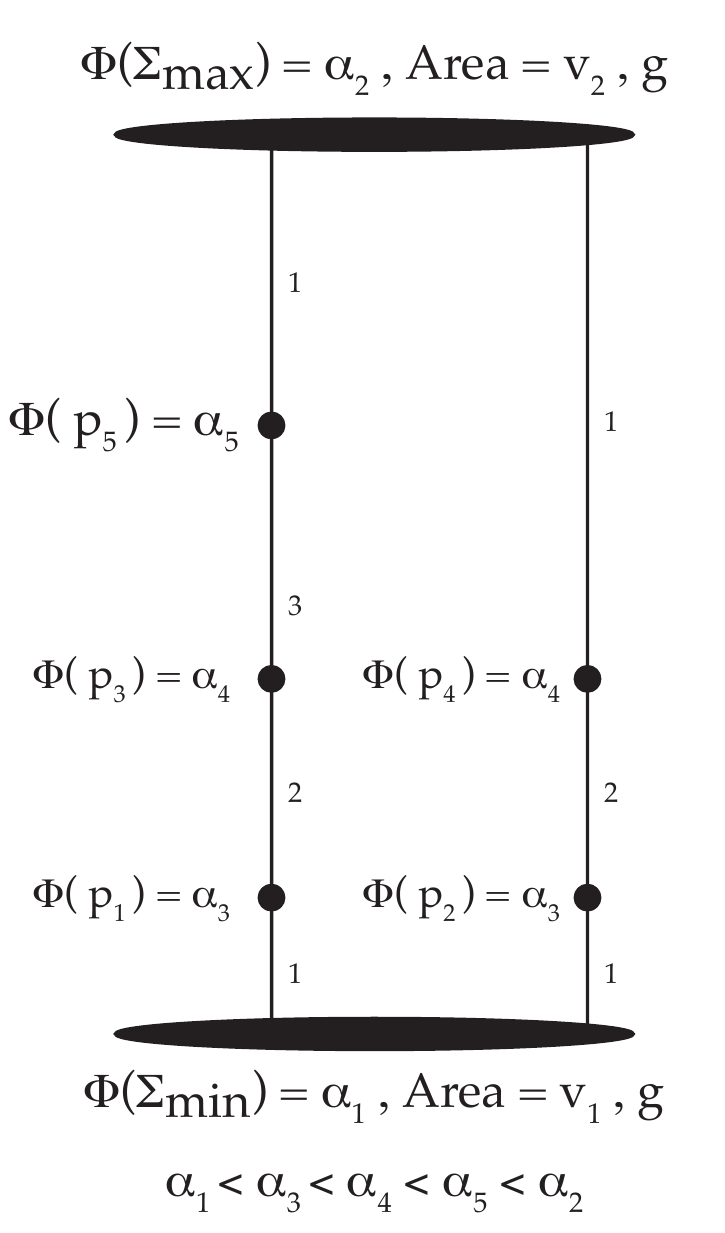}}\end{tabular}
 \hskip 1cm
\begin{tabular}{@{}c@{}}{\includegraphics[height=7cm]{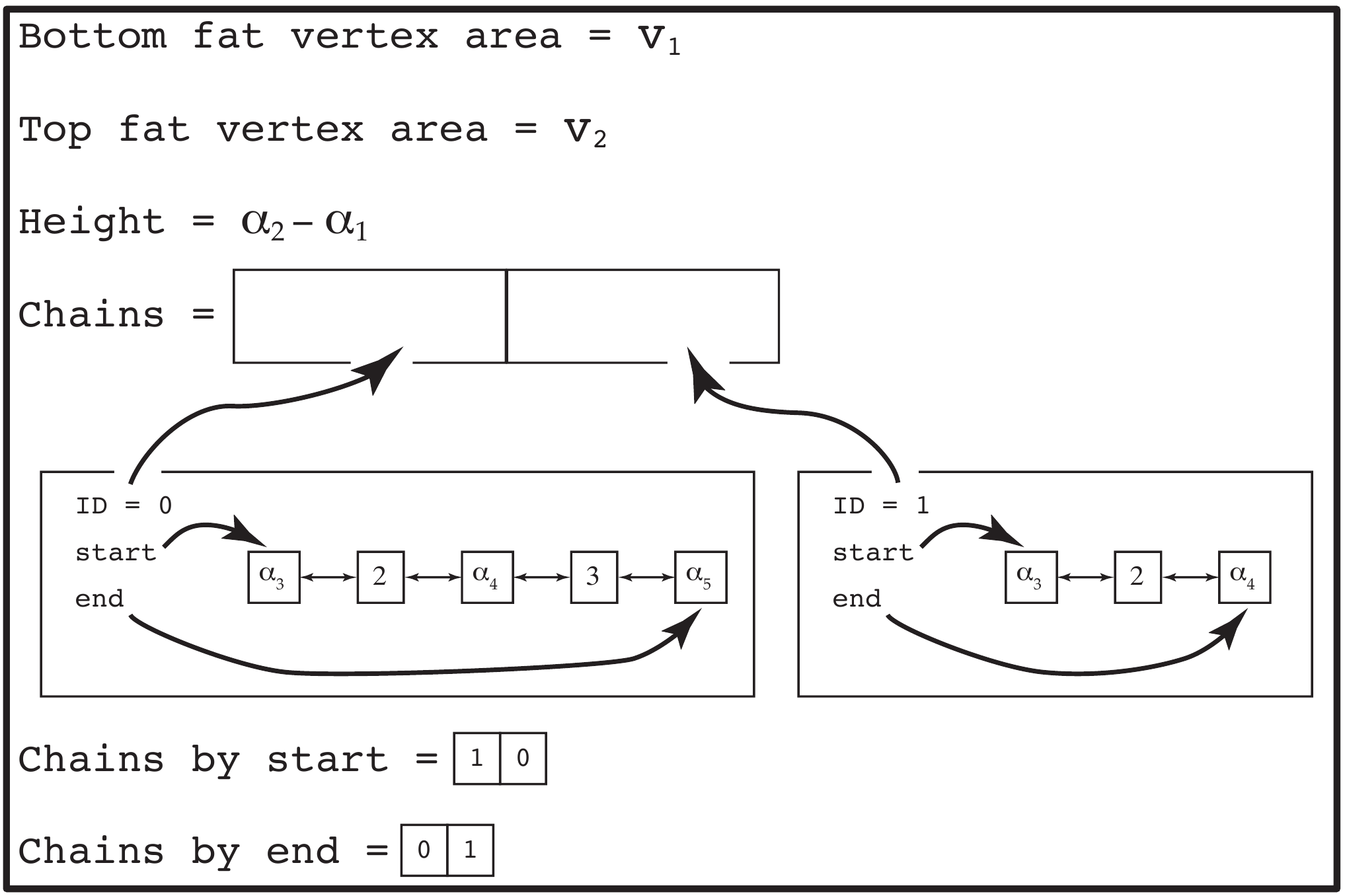}}\end{tabular}

\caption{{\bf On the left} is a decorated graph.
{\bf On the right} is the corresponding data structure.}
\label{fig:data structure}

\end{figure}

Another structure we use is referred to as \textbf{``Graphs Tree"}. It is a balanced binary search tree, storing a set of non-equivalent Graphs. Each node in the tree contains two fields - a key field and a field containing a linked list of non-equivalent Graphs of that key.
The key of a Graph is encoding the areas of the two fat vertices in the Graph - the larger area first, and the number of the Chains in the Graph. Notice that two Graphs may be equivalent only if they have the same key, meaning that they are stored in the same Graphs Tree node.
The operation of adding a Graph to the Graphs Tree would be done only if there is no equivalent Graph already kept in the Graphs Tree.

\subsection*{Algorithms}
The main algorithm counts the number of non-equivalent Graphs that can be produced as blowups of sizes $\delta_1,\dots,\delta_n$ from the
Graphs determined by $\lambda_B$ and $\lambda_F$.  It calls on the remaining algorithms defined in this section.

\begin{algorithm}[h]
\caption{ --- \textsc{CountGraphs} $(\lambda_F,\lambda_B;\delta_1,\delta_2,\dots,\delta_n)$.     \newline
$\circ$  \textit{Input}:  Two positive real numbers $\lambda_F, \lambda_B;$ and $n$ positive real numbers $\delta_1,\delta_2,\ldots,\delta_n$, 
in decreasing order. \newline 
$\circ$  \textit{Output}: The number of non-equivalent Graphs that can be created by performing the blowups sequentially, 
starting from the Graphs that are determined by $\lambda_B$ and $\lambda_F$.}\label{alg:CountGraphs}
	\begin{algorithmic}[1]
	         \State \textbf{Create} a Graphs Tree structure $gt$ (as described in the data structures section) of the Graphs defined by $\lambda_B$, $\lambda_F$ using \textsc{Graphs}$(\lambda_F, \lambda_B;)$.
	         
	         \State \textbf{Create} a Graphs Tree structure $t$ of all of the non-equivalent Graphs obtained from performing the blowups of sizes $\delta_1, \delta_2,\ldots, \delta_n$ sequentially, in every possible way, on each of the Graphs in $gt$, using \textsc{BlowupSet}$(gt,\delta_1, \delta_2,\ldots, \delta_n)$.
	         
	         \State \textbf{Return} the number of Graphs in $t$.

	\end{algorithmic}
\end{algorithm}

  \newpage

\begin{algorithm}[h]
\caption{ --- \textsc{Graphs} $(\lambda_F, \lambda_B;)$     \newline
$\circ$  \textit{Input}:  Two positive real numbers $\lambda_F, \lambda_B$  \newline 
$\circ$  \textit{Output}: Graphs Tree structure $gt$, storing the Graphs defined by $\lambda_B$ and $\lambda_F$}\label{alg:Graphs}
	\begin{algorithmic}[1]
	         \State \textbf{for} $i=0$ to $\left \lceil{\frac{\lambda_B}{\lambda_F}} \right \rceil-1$         
	         \State $\phantom{MO}$ \textbf{Create}  a Graph with the following data:\newline
$\phantom{MOVE}$ \texttt{bottom fat vertex area =} $\lambda_B+i\lambda_F$\newline
$\phantom{MOVE}$ \texttt{top fat vertex area =} $\lambda_B-i\lambda_F$\newline
$\phantom{MOVE}$ \texttt{Height =} $\lambda_F$\newline
$\phantom{MOVE}$ \texttt{Chains =} none %\newline
	         
\State $\phantom{MO}$ \textbf{Add} the Graph created to the Graphs Tree structure $gt$
\State {\bf Return} $gt$
	\end{algorithmic}
\end{algorithm}

\begin{algorithm}[h]
\caption{ --- \textsc{BlowupSet} $(gt, \delta_1, \delta_2,\ldots, \delta_n)$    \newline
$\circ$  \textit{Input}:  A Graphs Tree structure $gt$ and $n$ positive real numbers $\delta_1,  \delta_2, \ldots, \delta_n$ in decreasing order. \newline 
$\circ$  \textit{Output}:  Graphs Tree structure $t$, storing the non-equivalent Graphs that can be created by performing the blowups sequentially, starting from the Graphs in $gt$.}\label{alg:BlowupSet}
	\begin{algorithmic}[1]
\State \textbf{for} $i=1$ to $n$
\State $\phantom{MO}$ \textbf{Define} $t$ as a new empty Graphs Tree 
\State $\phantom{MO}$ \textbf{for} each Graph $g$ in $gt$
\State $\phantom{MOVE}$ \textbf{Create} all the Graphs that can be obtained by performing a blowup 

$\phantom{Mi}$ of size $\delta_i$  on $g$  (perform the blowup in every possible way in $g$), 

$\phantom{Mi}$ using \textsc{BlowupGraph}$(g,\delta_i)$.

\State $\phantom{MOVE}$ \textbf{Add} each Graph that have been created in the previous step 

$\phantom{Mi}$ to the Graphs Tree $t$ using \textsc{AddGraph}$(g,t)$ (an algorithm that

$\phantom{Mi}$  ensures that a Graph would be added if and only if there is no 

$\phantom{Mi}$ equivalent Graph in the Graphs Tree).
\State $\phantom{MO}$ \textbf{Set} $gt=t$
\State \textbf{Return} $gt$ 
	\end{algorithmic}
\end{algorithm}

  \newpage

\begin{algorithm}[h]
\caption{ ---  \textsc{BlowupGraph}$(g,\delta)$    \newline
$\circ$  \textit{Input}:  A Graph $g$ and a positive number $\delta$ \newline 
$\circ$  \textit{Output}:  A list $\ell ist$ of all the Graphs can be obtained by performing a blowup of size $\delta$ on the Graph $g$}\label{alg:BlowupGraph}
	\begin{algorithmic}[1]
\State \textbf{if} a blowup can be performed at the bottom fat vertex 

(that is,  if $\delta < \mathtt{bottomFatVertexArea}$) \textbf{then}

\State $\phantom{MO}$ \textbf{Create} a deep copy $g'$ of Graph $g$. Let $n$ be the number of Chains 

in $g$, make the size of the arrays in $g'$ (Chains, ChainsByStart, ChainByEnd) 

be $n+1$.

\State $\phantom{MO}$ \textbf{Subtract} $\delta$ from the area of the bottom fat vertex.
\State $\phantom{MO}$ \textbf{Add} a new Chain with one node $\delta$ to $g'$. Let $n$ be the number of Chains 

in $g$, the ID of the new Chain will be $n$ and a pointer to the new Chain

will be placed in Chains[$n$].

\State $\phantom{MO}$ \textbf{Use binary search and the sorting methods} defined in the data 

structure section to find the right place for the ID of the new 

Chain in the ChainsByStart  and ChainsByEnd arrays and place 

it there (move other IDs to the right if needed to ``make place").

\State $\phantom{MO}$ \textbf{Add} $g'$ to the $list$.

\State \textbf{if} a blowup can be performed on the top fat vertex ($\delta < \mathtt{topFatVertexArea}$) \textbf{then}

\State $\phantom{MO}$ \textbf{Do the same} as when a blowup can be performed at the bottom 

fat vertex, except subtract $\delta$ from the top fat vertex area and the new Chain 

will have the node height { \color{Black}{$-\delta$}}

\State \textbf{for} every Chain $chain$

\State $\phantom{MO}$ \textbf{for} every node $v$ that represents a vertex in $chain$

\State $\phantom{MOVE}$ \textbf{if} a monotone blowup can be made on the vertex \textbf{then}

\State $\phantom{MOVE IT}$ \textbf{Create} a new Chain replacing to $chain$ with a blowup on $v$, 

$\phantom{MOV}$ with the same ID.

\State $\phantom{MOVE IT}$ \textbf{Create} a deep copy $g'$ of the Graph $g$ and replace the copy of $chain$ 

$\phantom{MOV}$ with the new Chain that have been created.

\State $\phantom{MOVE IT}$ \textbf{Use binary search and the sorting methods} described above, 

$\phantom{MOV}$  to find the right place for the ID of the new Chain in the 

$\phantom{MOV}$ ChainsByStart and ChainsByEnd arrays and place it there 

$\phantom{MOV}$ (move other IDs to the right if needed to ``make place").

\State $\phantom{MOVE IT}$ \textbf{Add} $g'$ to $list$.
\State \textbf{Return} $list$
	\end{algorithmic}
\end{algorithm}

  \newpage

\begin{algorithm}[h]
\caption{ ---  \textsc{AddGraph}$(g,t)$   \newline
$\circ$  \textit{Input}:  A Graph $g$ and a Graph Tree structure $t$ \newline 
$\circ$  \textit{Output}:  The Graph $g$ would be added to the Tree structure $t$ if there is no equivalent Graph to $g$ in $t$}
\label{alg:AddGraph}
	\begin{algorithmic}[1]

\State \textbf{Get} the node from $t$ with the same key as the key of the Graph $g$.
\State \textbf{if} there is no such node in $t$ \textbf{then}

\State $\phantom{MO}$ \textbf{add} a new node to $t$ with the key of $g$. 

\textbf{Set} the list of Graphs of this node to contain $g$.

\State \textbf{else}

\State $\phantom{MO}$ \textbf{for} each Graph $g_2$ in the list of Graphs of the received node

\State $\phantom{MOVE}$ \textbf{if} $g$ is equivalent to $g_2$, (to be checked using 
\textsc{AreEquivalent}$(g,g_2)$), \textbf{then}

\State $\phantom{MOVE}$ \textbf{return}

\State $\phantom{MO}$ \textbf{add} $g$ to the list of Graphs in the node that matches to its key.

	\end{algorithmic}
\end{algorithm}

\begin{algorithm}[h]
\caption{ ---  \textsc{AreEquivalent}$(g_1,g_2)$  \newline
$\circ$  \textit{Input}:  A Graph $g_1$ and a Graph $g_2$ . Both $g_1$ and $g_2$ have the same key according to the Graphs Tree structure \newline 
$\circ$  \textit{Output}:  True, if the two Graphs are equivalent  and  False  otherwise}
\label{alg:AreEquivalent}
	\begin{algorithmic}[1]
\State   \textbf{if} $g_1$.bottomFatArea $=$ $g_1$.topFatArea

\State $\phantom{MO}$  \textbf{Return} \textsc{AreTheSame}$(g_1, g_2)$ or \textsc{AreReflection}$(g1,g2)$

\State \textbf{else}
\State $\phantom{MO}$  \textbf{if} ($g_1$.bottomFatArea=$g_2$.bottomFatArea)

\State $\phantom{MOVE}$  \textbf{Return} \textsc{AreTheSame}$(g_1,g_2)$

\State $\phantom{MO}$ \textbf{else }

\State $\phantom{MOVE}$ \textbf{Return} \textsc{AreReflection}$(g_1,g_2)$ 
	\end{algorithmic}
\end{algorithm}

\begin{algorithm}[h]
\caption{ ---  \textsc{AreTheSame}$(g_1,g_2)$  \newline
$\circ$  \textit{Input}:   A Graph $g_1$ and a Graph $g_2$. Both $g_1$ and $g_2$ have the same key according to the Graphs Tree structure. $g_1$ and $g_2$ fat vertices area are equals respectively. \newline 
$\circ$  \textit{Output}:  True, if the two Graphs are exactly the same; False otherwise.}
\label{alg:AreTheSame}
	\begin{algorithmic}[1]

\State \textbf{for} $i =0$ to the number of Chains in $g_1$ and $g_2$

\State $\phantom{MO}$ \textbf{if} not $g_1$.Chain[$g_1$.ChainsByStart[$i$]] equals $g_2$.Chain[$g_2$.ChainsByStart[$i$]]

\State $\phantom{MOVE}$ \textbf{Return} false

\State \textbf{Return} true
	\end{algorithmic}
\end{algorithm}

\begin{algorithm}[h]
\caption{ ---  \textsc{AreReflection}$(g_1,g_2)$  \newline
$\circ$  \textit{Input}:    Graph $g_1$ and a Graph $g_2$ . Both $g_1$ and $g_2$ have the same key according to the Graphs Tree structure. $g_1$ first fat area equals $g_2$ last fat area and vice versa. \newline 
$\circ$  \textit{Output}:  True, if the two Graphs are a reflection of one another; False otherwise.}
\label{alg:AreReflection}
	\begin{algorithmic}[1]
\State \textbf{for} $i =0$ to the number of Chains in $g_1$ and $g_2$

\State $\phantom{MO}$ \textbf{if not} $g_1$.Chain[$g_1$.ChainsByStart[$i$]] is a 

reflection of $g_2$.Chain[$g_2$.ChainsByEnd[$i$]] (check node by node, 

one Chain from the start and the other from the end, labels should be 

equal and a vertex should be equal to the height of the Graph minus 

the corresponding one.) \textbf{then}

\State $\phantom{MOVE}$ \textbf{Return} false
\State \textbf{Return} true
	\end{algorithmic}
\end{algorithm}

\end{document}